\documentclass[12pt]{amsart}

\usepackage{amssymb,mathrsfs,bm,amsmath,amsthm,color,mathtools,wasysym}

\usepackage{enumerate}
\usepackage[centering]{geometry}
\allowdisplaybreaks
\theoremstyle{plain}
\newtheorem{thm}{Theorem}[section]

\newtheorem{prop}[thm]{Proposition}

\newtheorem{lem}[thm]{Lemma}

\theoremstyle{remark}
\newtheorem{rem}{Remark}
\numberwithin{equation}{section}

\DeclareMathOperator{\hdim}{\dim_H}

\newcommand{\dist}{\mathrm{dist}}

\newcommand{\Q}{\mathbb Q}
\newcommand{\N}{\mathbb N}

\newcommand{\R}{\mathbb R}

\newcommand{\lm}{\mathcal L}

\begin{document}
	\title{On Dirichlet non-improvable numbers and shrinking target problems}

\author{Qian Xiao}
\address{School of Mathematics and Statistics, Southwest University, Chongqing,400715, P.R. China}
\email{xiaoqianmath@163.com}



\begin{abstract}
	In one-dimensional Diophantine approximation, the Diophantine properties of a real number are characterized by its partial quotients, especially the growth of its large partial quotients. Notably, Kleinbock and Wadleigh [Proc. Amer. Math. Soc. 2018] made a seminal contribution by linking the improvability of Dirichlet's theorem to the growth of the product of consecutive partial quotients. In this paper, we extend the concept of Dirichlet non-improvable sets within the framework of shrinking target problems. Specifically, consider the dynamical system $([0,1), T)$ of continued fractions. Let $\{z_n\}_{n \ge 1}$ be a sequence of real numbers in $[0,1]$ and let $B > 1$. We determine the Hausdorff dimension of the following set:
	\[
	\begin{split}
		\{x\in[0,1):|T^nx-z_n||T^{n+1}x-Tz_n|<B^{-n}\text{ infinitely often}\}.
	\end{split}
	\]

\end{abstract}

	\maketitle

\section{Introduction}
The central question in Diophantine approximation is: how well can a given real number be approximated by rational numbers?
In one-dimensional settings, continued fraction serves as an important tool for this purpose, providing an algorithmic solution for finding the best rational approximation of a given real number.
The continued fraction can be computed by the Gauss transformation $T: [0,1) \to [0,1)$ defined as
\[
T(0)=0,~~T(x)= 1/x-\lfloor 1/x\rfloor ~\text{if}~ x \in (0,1),
\]
where $\lfloor 1/x\rfloor$ is the integer part of  $1/x$.
For $x \in (0,1)$, put $a_1 (x) =\lfloor 1/x\rfloor $ and $a_{n+1}(x)=\lfloor 1/T^{n}(x)\rfloor=a_1(T^n (x))$ for $n \geq 1$.
Then $x \in (0,1)$ can be written as the continued fraction expansion
\begin{equation}
x= \cfrac{1}{a_1(x)+\cfrac{1}{a_2(x)+\ddots}}=:[a_1(x),a_2(x), \cdots]
\end{equation}
where $a_1(x), a_2(x), \dots$ are positive integers, and called the {\it partial quotients} of $x$. Let $x = [a_1(x), a_2(x), \dots]$ be its continued fraction expansion and the truncation $p_n(x)/q_n(x)=[a_1(x),a_2(x),\dots,a_n(x)]$ be its $n$th {\it convergent}.
% Occasionally, we leave out the dependence on $x$ and simply write $a_n$ and $p_n/q_n$.
The continued fraction expansion of a real number is widely recognized for its significant role in studying one-dimensional homogeneous Diophantine approximation.
This can be inferred from the following two fundamental results.

\begin{thm}
Optimal rational approximation of the convergent:
\begin{equation*}
	\min_{1 \leq q \leq q_n(x), p \in \mathbb{Z}} \bigg| x - \frac{p}{q} \bigg| =\bigg| x - \frac{p_n(x)}{q_n(x)} \bigg|,~\min_{1 \leq q <q_n(x)} \| qx \| =\| q_{n-1}(x) x \|,
\end{equation*}
where $ \| \cdot \|$ denotes the distance to the nearest integer.

Legendre's theorem:
\begin{equation*}
	\bigg| x - \frac{p}{q} \bigg| < \frac{1}{2q^2}  \Longrightarrow \frac{p}{q} =\frac{p_n(x)}{q_n(x)}, ~\text{for~some} ~n \geq 1.
\end{equation*}
\end{thm}

Building upon these two results, the Diophantine properties of a real number are largely characterized by its partial quotients, especially the growth of its large partial quotients within the consideration of the current paper.

The metrical theory of continued fractions, which concerns the size
 (in terms of measure or Hausdorff dimension, etc.) of the sets obeying some restrictions on their partial quotients, is an important subject in studying continued fractions. One focus is the study of the following sets
\begin{equation*}
	E_m (B) : = \{x \in [0,1): a_n(x) a_{n+1}(x) \cdots a_{n+m-1}(x) \geq B^n~\text{i.o.}\},
\end{equation*}
where $m\in \N$,  $B>1$ and `` i.o." stands for `` infinitely often". It is worth noting that the sets $E_1(B)$ and $E_2(B)$ are related to homogeneous Diophantine approximation and Dirichlet non-improvable numbers (see \cite[Lemma 2.2]{KW2018}), respectively. The Hausdorff dimension of $E_m(B)$ is completely given in the following result:
\begin{thm}[{\cite{WW2008}} and {\cite[Theorem 1,7]{HWX2020}}]
We have
\begin{equation*}
	\hdim E_m(B)=\inf \{ s: P(T, -f_m(s) \log B - s \log |T'|) \leq 0 \},
\end{equation*}
 where $\hdim$ denotes the Hausdorff dimension, $P(T,\cdot)$ is a pressure function defined in Section \ref{ss:pressure}, and $f_m(s)$ is given by the following iterative formulae:
\begin{equation*}
f_1(s) =s,~f_{k+1} (s) = \frac{sf_k(s)}{1-s+f_k(s)},~k \geq 1.
\end{equation*}
\end{thm}

There are many studies on Hausdorff dimensions of the sets related to $E_m(B)$, for example \cite{BBH2020-ETDS, BBH2020, Good1941,HWX2020,HLS2024, HS2023, LWX2023, Luczak1997,WW2008}.

Since the partial quotients can be obtained through Gauss map, the theory also has close connections with dynamical systems and ergodic theory. Note that
\[a_n(x)=a_1(T^{n-1}x)=\bigg\lfloor\frac{1}{T^{n-1}x}\bigg\rfloor\in \bigg[\frac{1}{2T^{n-1}x},\frac{1}{T^{n-1}x}\bigg],\]
and so
\begin{align}\label{eq:1}
	a_n(x)\ge B^n \Longrightarrow T^{n-1}x\le B^{-n} \quad \text{and}\quad
T^{n-1}x\le B^{-n}	\Longrightarrow a_n(x)\ge B^n/2.
\end{align}
In other words, the $n$th partial quotient being sufficiently large corresponds to $T^{n-1}x$ being sufficiently close to 0. From this simple observation, Li, Wang, Wu and Xu \cite{LWWX2014} came to the following generalization of $E_1(B)$:
\[E_1(\{z_n\}_{n\ge 1},B):=\{x\in[0,1]:|T^nx-z_n|\le B^{-n}~\text{ i.o.}\},\]
where $\{z_n\}_{n\ge 1}$ is a sequence of real numbers in $ [0,1] $.
They further showed that the Hausdorff dimension of this set is the same as that of $E_1(B)$.  It is not difficult to deduce from \eqref{eq:1} that $E_1(\{z_n\}_{n\ge 1},B)$ almost returns to $E_1(B)$ if $z_n\equiv 0$ for all $n\ge 1$. The study of the metrical property of $E_1(\{z_n\}_{n\ge 1},B)$ is also referred to as {\em shrinking target problem}, which is initially introduced by Hill and Velani \cite{HV1997}, and has gained much attention these decades.
See \cite{AB2021, BR2018, BW2014, HV1999, KLR2022, LLVZ2023} and reference therein.

Following this kind of philosophy and starting from yet another observation
\[
a_n(x)a_{n+1}(x)\ge B^n\Longrightarrow T^{n-1}x\cdot T^nx\le B^{-n},
\]
and
\[
T^{n-1}x\cdot T^nx\le B^{-n}\Longrightarrow a_n(x)a_{n+1}(x)\ge B^n/4,
\]
we introduce the following generalization of $E_2(B)$:
\[E_2(\{z_n\}_{n\ge 1},B):=\{x\in[0,1):|T^nx-z_n||T^{n+1}x-Tz_n|<B^{-n}~\text{ i.o.}\}\]
with $\{z_n\}_{n\ge 1}$ and $B>1$ given above.
 The Hausdorff dimension of the set $ E_2(\{z_n\}_{n\ge 1},B) $ is completely determined in the current work. For each $ n\ge 1 $, let us define three quantities as follows
\begin{equation}\label{eq:sni}
	\begin{split}
		s_{n,1}&=\inf\bigg\{s\in[0,1]:\sum_{a_1,\dots,a_n\in\N}\frac{1}{q_n(a_1,\dots,a_n)^{2s}B^{ns^2}}\le 1\bigg\},\\
		s_{n,2}&=\inf\bigg\{s\in[0,1]:\sum_{a_1,\dots,a_n\in\N}\frac{a_1(z_n)^{1-s}}{q_n(a_1,\dots,a_n)^{2s}B^{ns}}\le 1\bigg\},\\
		s_{n,3}&=\inf\bigg\{s\in[0,1]:\sum_{a_1,\dots,a_n\in\N}\frac{1}{q_n(a_1,\dots,a_n)^{2s}a_1(z_n)^sB^{ns/2}}\le 1\bigg\}.
	\end{split}
\end{equation}
We adopt the convention that $ a_1(0)=+\infty $, in this case, set $ s_{n,2}=1 $ and $ s_{n,3}=0 $. The $ n $th pre-dimensional number is defined by
\begin{equation}\label{eq:sn}
	s_n=\begin{cases}
		s_{n,1}\quad&\text{if }s_{n,1}\le s_{n,2},\\
		\max \{ s_{n,2},s_{n,3} \}\quad&\text{if }s_{n,1}> s_{n,2}.
	\end{cases}
\end{equation}

\begin{thm}\label{t:main}
	We have
	\[
	\hdim E_2(\{z_n\}_{n\ge 1},B)=\limsup_{n\to\infty} s_n : = s^*.
	\]
\end{thm}

\begin{rem}
	We consider the approximation of $ |T^n x - z_n| |T^{n+1} x - T z_n| $ instead of $ |T^n x - z_n| |T^{n+1} x - y_n| $ for the following reasons.
	The latter focuses on the interaction between the approximation of $ T^n x $ to $ z_n $ and that of $ T^{n+1} x $ to $ y_n $. Assume for the moment that $ T^n x $ and $ T^{n+1} x $ are sufficiently close to $ z_n $ and $ y_n $, respectively. Then, most prefixes of the continued fraction expansions of $T^nx$ and $z_n$ are the same. The same applies to $T^{n+1}x$ and $y_n$. These imply that $y_n$ is sufficiently close to $Tz_n$, and so is $ |T^n x - z_n| |T^{n+1} x - T z_n| $ to $ |T^n x - z_n| |T^{n+1} x - y_n| $. Therefore, our setting is not as restrictive as it may seem.

\end{rem}
\begin{rem}
	The Hausdorff dimension of $E_2(\{z_n\}_{n\ge 1},B)$ depends on the location of $z_n$, more precisely, on $a_1(z_n)$, which does not happen with $E_1(\{z_n\}_{n\ge 1},B)$. Let us illustrate it with two simple examples. If $z_n\equiv 0$ for all $n\ge 1$, then $E_2(\{z_n\}_{n\ge 1},B)$ almost reduces to $E_2(B)$. This motivates the definition of $s_{n,1}$. If $z_n\equiv [1,1,\dots]$ for all $n\ge 1$, then by a simple fact from the theory of continued fraction,
	\[\begin{split}
		&|T^nx-z_n|\text{ is small enough}\\
		\Longrightarrow &|T^nx-z_n|=|T^{n+1}x-Tz_n|, \text{ up to a multiplicative constant}.
	\end{split}\]
	Therefore, up to a multiplicative constant
	\[|T^nx-z_n||T^{n+1}x-Tz_n|<B^{-n}\Longrightarrow |T^nx-z_n|<B^{-n/2}.\]
	This motivates the definition of $s_{n,3}$. As for $s_{n,2}$, it can be interpreted as follows: As $z_n$ varies from $1$ to $0$, the optimal cover of $E_2(\{z_n\}_{n\ge 1},B)$ changes, leading to the definition of $s_{n,2}$.
\end{rem}

The structure of this paper is as follows.
In Section 2, we recall several notions and elementary properties of continued fractions.
We prove the upper bound and lower bound of Hausdorff dimension of the set $E_2(\{z_n\}_{n\ge 1},B)$ in Sections 3 and 4, respectively.

%In Section 3, we prove the upper bound of Hausdorff dimension of  $E_2(\{z_n\}_{n\ge 1},B)$ and in Section 4, we estimate the lower bound of Hausdorff dimension of  $E_2(\{z_n\}_{n\ge 1},B)$.

\section{Preliminaries}

In this section, we recall some basic properties of continued fractions and pressure functions, as well as establishing some basic facts. Throughout, for two variables $f$ and $g$, the notation $f\ll g$ means that $f\le cg$ for some
unspecified constant $c$, and the notation $f\asymp g$ means that $f\ll g$ and $g\ll f$.
For a set $A$, $|A|$ stands for the diameter of $A$.
We use $\mathcal L$ to represent the Lebesgue measure.

\subsection{Continued fraction}

It is well-known that if $x \in (0,1)$ is a rational number, the expansion of $x$ is finite; if $x \in (0,1)$ is {\color{red} an} irrational number, the expansion of $x$ is infinite.
A finite {\color{red}truncation} on the expansion of $x$ gives rational fraction $p_n(x)/q_n(x): =[a_1(x), a_2(x), \ldots, a_n(x)]$, which is called the $n$th convergents of $x$.
With the conventions
 \[
 p_{-1}=1,\ q_{-1}=0,\ p_0=0~ \text{and} ~q_0=1,
 \]
 the sequence  $p_n=p_n(x)$ and $q_n =q_n(x)$ can be given by the following recursive relations
\begin{equation}\label{convergents}
	p_{n+1}=a_{n+1}(x)p_n+p_{n-1},~q_{n+1}=a_{n+1}(x)q_n+q_{n-1}.
 \end{equation}
Clearly, $q_n(x)$ is determined by $a_1(x), \ldots, a_n(x)$.
So we may write $q_n(a_1(x),\ldots, a_n(x))$.
When no confusion is likely to arise, we write $a_n$ and $ q_n$, respectively, in place of $a_n(x)$ and $q_n(x)$ for simplicity.

For an integer vector $(a_1,a_2,\ldots,a_n) \in \mathbb{N}^n$ with $ n\geq 1$, denote
\begin{equation}
	I_n (a_1, a_2, \ldots,a_n):=\{x \in [0,1): a_1(x)=a_1, a_2(x)=a_2, \ldots, a_n(x)=a_n\}
\end{equation}
for the corresponding $n$th level {\em cylinder}, i.e. the set of all real numbers in $[0,1)$ whose continued fraction expansions begin with $(a_1, \ldots, a_n)$.

We will frequently use the following well-known properties of continued fraction expansion. They are explained in the standard texts \cite{IK2002, Khintchine1963}.

\begin{prop}\label{p:basic property}
For any positive integers $a_1, \ldots, a_n$, let $p_n = p_n(a_1, \ldots, a_n)$ and $q_n = q_n (a_1, \ldots, a_n)$ be defined recursively by (\ref{convergents}).
\begin{enumerate}
	\item $q_n \geq 2^{(n-1)/2}$ and
	\begin{equation}
		\prod_{i=1}^n a_i \leq q_n \leq \prod_{i=1}^n (a_i +1) \leq  \prod_{i=1}^{n} 2a_i.
	\end{equation}

	\item It holds that
	\[
	q_n \leq (a_n +1) q_{n-1},~~~ 1 \leq \frac{q_{n+k} (a_1, \ldots, a_n, \ldots, a_{n+k} ) }{ q_n(a_1, \ldots, a_n) q_k (a_{n+1}, \ldots, a_{n+k})} \leq 2,
	\]
	\[
	\frac{a_k +1}{2} \leq \frac{q_n(a_1, \ldots, a_n)}{q_{n-1} (a_1, \ldots, a_{k-1}, a_{k+1}, \ldots, a_n)} \leq a_k+1.
	\]
	\item
		$I_n (a_1, \ldots, a_n )= \begin{cases}
	\left[\cfrac{p_n}{q_n}, \cfrac{p_n+p_{n-1}}{q_n +q_{n-1}} \right), ~~\text{if}~ n ~\text{is~ even},\\
			\left(\cfrac{p_n+p_{n-1}}{q_n +q_{n-1}},\cfrac{p_n}{q_n} \right],~~ \text{if}~ n~ \text{is~odd}. \end{cases}$

\noindent The length of $I_n (a_1, \ldots, a_n)$ is given by
\begin{equation}
	\cfrac{1}{2 q_n ^2} \leq | I_n (a_1, \ldots, a_n ) | =\cfrac{1}{q_n (q_n + q_{n-1})} \leq \cfrac{1 }{q_n^2}.
\end{equation}

\end{enumerate}
\end{prop}

The next proposition describes the positions of cylinders $I_{n+1}$ of level $n+1$ inside the $n$th level cylinder $I_n$.

\begin{prop} \label{C-D}
Let $I_n =I_n (a_1, \ldots, a_n)$ be {\color{red}an} $n$th level cylinder, which is partitioned into sub-cylinders $\{I_{n+1}(a_1, \ldots, a_n, a_{n+1}): a_{n+1} \in \mathbb{N}\}$. When $n$ is odd, these sub-cylinders are positioned from left to right, as $a_{n+1}$ increases from $1$ to $\infty$; when $n$ is even, they are positioned from right to left.
\end{prop}

%We end this part with a dimensional result in continued fractions due to $\cdots$

\subsection{Pressure function}\label{ss:pressure}

Pressure function is an appropriate tool in dealing with dimension problems in infinite conformal iterated function systems.
We {\color{red}recall} that the pressure function with a continuous potential can be approximated by the pressure functions restricted to the sub-systems in continued fractions. For more information on pressure functions, we refer the readers to \cite{HMU2002,  MU1996, MU1999}.

Let ${\mathbb A}$ be a finite or infinite subset of $\mathbb{N}$, and define
\begin{equation*}
	X_{\mathbb A}= \{ x \in [0,1): a_n (x) \in {\mathbb A}~ \text{for~all}~n \geq 1 \}.
\end{equation*}
Then, with the Gauss map $T$ restricted to it,  $X_{\mathbb A}$ forms a dynamical system.
 The pressure function restricted to this sub-system $(X_{\mathbb A}, T)$ with potential $\phi : [0,1) \to \mathbb{R}$ is defined as
 \begin{equation}\label{def:pressure}
 	P_{\mathbb A} (T, \phi) = \lim_{n \to \infty} \cfrac{1}{n} \log \sum_{(a_1, \ldots, a_n ) \in {\mathbb A}^n } \sup_{x \in X_{\mathbb A}} e^{S_n \phi([a_1, \ldots, a_n +x])},
 \end{equation}
where $S_n \phi (x) $ denotes the ergodic sum $\phi(x)+\cdots+\phi(T^{n-1} x)$.
When ${\mathbb A}=\mathbb{N}$, we write $P(T, \phi)$ for $P_{\mathbb{N}} (T, \phi)$.

The $n$th variation $Var_n(\phi)$ of $\phi$ is defined as
\begin{equation*}
	Var_n (\phi) := \sup \big\{ |\phi(x)-\phi(y)|:  I_n (x) =I_n (y)\big\}.
\end{equation*}

The existence of the limit in equation \eqref{def:pressure} is due to the following result.

\begin{prop}[{\cite{LWWX2014} and \cite{MU1996}}]
Let $\phi : [0,1) \to \mathbb{R}$ be a real function with $Var_1(\phi) < \infty$ and $Var_n(\phi) \to 0$ as $n \to \infty$.
Then the limit defining $P_{\mathbb A}(T, \phi )$ in \eqref{def:pressure} exists and the value of $P_{\mathbb A}(T,\phi)$ remains the same even without taking the supremum over $x \in X_{\mathbb A}$ in \eqref{def:pressure}.
\end{prop}

The following proposition states a continuity property of the pressure function when the continued fraction system $( [0,1), T) $ is approximated by its sub-systems $(X_{\mathbb A}, T)$.
%For the detailed proof, see \cite{HMU2002}.

\begin{prop}[{\cite[Proposition 2]{HMU2002}}]\label{p:property of s}
Let $\phi : [0,1) \to \mathbb{R}$ be a real function with $Var_1(\phi) < \infty $ and $Var_n(\phi) \to 0$ as $n \to \infty$.
 We have
 \begin{equation}
 	P(T, \phi) = P_{\mathbb{N}} (T, \phi)=\sup \{ P_{\mathbb A}(T, \phi) : {\mathbb A}~\text{is~a~finite~subset~of}~ \mathbb{N}\}.
 \end{equation}
% As a consequence, we have
 %\begin{equation}
 %	s(\mathbb{N}, \phi) = \sup \{ s({\mathbb A}, \phi) : {\mathbb A}~\text{is~a~finite~subset~of}~ \mathbb{N} \},
 %\end{equation}
 %where $s(\mathbb {\mathbb A}, \phi) $ is defined as $\inf \{ s \geq 0: P_{\mathbb A}(T, s \phi) \leq 0\}$ for $\mathbb A \subset \mathbb{N}$.
\end{prop}
The potential functions related to the dimension of $E_2(\{z_n\}_{n\ge 1}, B)$ will be taken as the following forms:
\begin{enumerate}
 	\item $\phi_1(x)=-s\log|T'(x)|-s^2\log B$ corresponding to the pre-dimensional number $s_{n,1}$ (see \eqref{pressure1});
	\item $\phi_2(x)=-s\log|T'(x)|-s\log B+(1-s)\alpha$ for some $\alpha>0$ corresponding to the pre-dimensional number $s_{n,2}$ (see \eqref{pressure2});
	\item $\phi_3 (x)=-s\log|T'(x)|-(s/2)\log B-s\beta$ for some $\beta \ge 0$ corresponding to the pre-dimensional number $s_{n,3}$ (see \eqref{pressure3}).
\end{enumerate}

%From now on, the set ${\mathbb A}$ is always taken to be a finite subset of $\N$.
%Let
%\[
%s_i({\mathbb A}) =\inf \{ s \geq 0: P_{\mathbb A}(T, \phi_i) \leq 0\}.
%\]
%% and $s(\ell, {\mathbb A})$ be the unique solution $s \ge 0$ to
%%\[
%% \sum_{\bm a\in {\mathbb A}^\ell} \frac{1}{q_\ell(\bm a)^{2s} B^{\ell s^2}} =1.
%%\]
%%For any fixed $M \in \N$, when $ = \{ 1, 2, \ldots, M\}$, we write  $s(\ell, M)$ for $s(\ell, {\mathbb A})$ and
%
%
%We apply Proposition 2.4 to the potential $\phi_i~(i=1,2,3)$ .
%It is clear that  satifies the variation condition.
%\begin{cor}
%	We have
%\begin{equation}
%s_i(\mathbb{N}) = \sup \{ s_i({\mathbb A}) : {\mathbb A}~\text{is~a~finite~subset~of}~ \mathbb{N} \}.
%\end{equation}
%\end{cor}
%
%The following corollary illustrates some relationships between the pre-dimensional numbers and the continuities of them in different paraters (see \cite{MU1996} also \cite{WW2008}).
%
%\begin{cor}[\cite{MU1996} ]
%For any $M \in \N$, we have
%\[
%\lim_{n \to \infty} s_n(M,\phi) =s(M, \phi),~~\lim_{n \to \infty} s_n(\N, \phi) =s(\N, \phi), ~~\lim_{M \to \infty} s(M, \phi) = s(\N, \phi).
%\]
%\end{cor}

\subsection{Basic facts}

Some lemmas will be established in this subsection for future use. The first one involves a summation that naturally appearing in the proof of the upper bound of $E_2(\{z_n\}_{n\ge 1},B)$.

%\begin{lem}\label{l:summation1}
%	Suppose that $a\ge B^{nt}$ for some $t>0$. Then, for any $s>0$, the following holds
%	\begin{equation}
%	\sum_{b\le B^{nt}/2 }\frac{a^{s}}{b^{s}|a-b|^{s}} \asymp B^{nt(1-s)}.
%	\end{equation}
%\end{lem}
%
%
%\begin{proof}
%	Since $a\ge B^{nt}$, for any $b\le B^{nt}/2$ it holds that
%	\[\frac{1}{2} a \leq |a-b| \leq a.\]
%	Using these two inequalities, we have
%	\[\sum_{b\le B^{nt}/2 }\frac{a^{s}}{b^{s}|a-b|^{s}} \asymp \sum_{b\le B^{nt}/2 }\frac{1}{b^{s}}\asymp \int_{1}^{B^{nt}/2}\frac{1}{x^s}dx\asymp B^{nt(1-s)}.\qedhere\]
%\end{proof}

\begin{lem}\label{l:summation2}
	For any $a\in\N^{+}$ and $t>1/2$, the following holds
	\begin{equation}
\sum_{b \in \N^{+}\setminus\{a\}}\frac{a^t }{b^t|a-b|^t}
	\asymp a^{1-t}.
	\end{equation}
\end{lem}

\begin{proof}
We are going to decompose $\{b\in\N: b\ne a\}$ appearing in the left summation into four blocks, and in each block using a particular estimate. Write
	\[\{b\in\N:b\ne a\}=A+B+C+D,\]
	where
	\begin{align*}
		A&=\{b\in\N:1\le b\le a/2\},& B&=\{b\in\N:a/2\le b<a\},\\
		C&=\{b\in\N:a< b\le 2a\},& D&=\{b\in\N:b>2a\}.
	\end{align*}

For the block $A$, since $1 \leq b \leq a/2$ for any $b\in A$, one has $a/2 \le |a-b| < a$. Thus,
	\[\sum_{b\in A}\frac{a^t }{b^t|a-b|^t}
 \asymp \sum_{b\in A}\frac{1 }{b^t}\asymp \int_1^{a/2} \frac{1}{x^t }  dx\asymp a^{1-t}.\]

For the block $B$, using $ a/2\leq b< a$ and both $a$ and $b$ are integers, we have
	\[	\sum_{b\in B}\frac{a^t }{b^t|a-b|^t}
	\asymp \sum_{b\in B}\frac{1 }{|a-b|^t}=\sum_{1\le c\le a/2}\frac{1}{c^t}\asymp a^{1-t}.\]

	The estimation for the block $C$ is similar to $B$, we omit the details.

For the {\color{red}block} $D$, since $b>2a$ for any $b\in D$, it follows that $ b/2 \leq   |a-b|< b$. So, by the condition $t>1/2$,
	\[\sum_{b\in D}\frac{a^t }{b^t|a-b|^t}
\asymp a^t\cdot\sum_{b\in D}\frac{1 }{b^{2t}}\asymp a^t\cdot\int_{2a}^{\infty} \frac{1}{x^{2t}}dx\asymp a^{1-t}.\]
Combine the above four estimations,  and the proof is completed.
\end{proof}

Next, we explore some properties of the pre-dimension numbers $s_{n,i}$ and their relationship to $a_1(z_n)$.
Recall the definitions of $s_{n,i}$ in \eqref{eq:sni}.
%\begin{equation*}
%	\begin{split}
%		s_{n,1}&=\inf\bigg\{s\in[0,1]:\sum_{a_1,\dots,a_n\in\N}\frac{1}{q_n(a_1,\dots,a_n)^{2s}B^{ns^2}}\le 1\bigg\},\\
%		s_{n,2}&=\inf\bigg\{s\in[0,1]:\sum_{a_1,\dots,a_n\in\N}\frac{a_1(z_n)^{1-s}}{q_n(a_1,\dots,a_n)^{2s}B^{ns}}\le 1\bigg\},\\
%		s_{n,3}&=\inf\bigg\{s\in[0,1]:\sum_{a_1,\dots,a_n\in\N}\frac{1}{q_n(a_1,\dots,a_n)^{2s}a_1(z_n)^sB^{ns/2}}\le 1\bigg\}.
%	\end{split}
%\end{equation*}

\begin{lem}\label{l:>1/2}
	Let $n\in \N$ and $f(s)=\sum_{a_1,\dots,a_n\in\N}q_n(a_1,\dots,a_n)^{-2s}$. We have
	\[f(s)<\infty \Longleftrightarrow s>1/2.\]
	Moreover, $f(s)$ is continuous on $s>1/2$ and goes to infinity as $s\downarrow 1/2$. Consequently,
	\begin{enumerate}
		\item $s_{n,i}>1/2\quad\text{for $i=1,2,3$}.$
		\item $s_{n,1}$ satisfies
		\[\sum_{a_1,\dots,a_n\in\N}\frac{1}{q_n(a_1,\dots,a_n)^{2s_{n,1}}B^{ns_{n,1}^2}}=1.\]
		Similar arguments apply to $s_{n,2}$ and $s_{n,3}$, with the summations being replaced in the obvious way according to their definitions.
	\end{enumerate}
\end{lem}
\begin{proof}
	The proofs follow the ideas in \cite[Lemma 2.6]{WW2008} and \cite[Lemma 3.2]{HMU2002}.

	By Proposition \ref{p:basic property} (1), for any $s>0$
	\begin{align}
		f(s)&=\sum_{a_1,\dots,a_n\in\N}q_n(a_1,\dots,a_n)^{-2s}\asymp \sum_{a_1,\dots,a_n\in\N}\prod_{i=1}^{n}a_i^{-2s}=\bigg(\sum_{a\in\N} a^{-2s}\bigg)^n\notag\\
		&\asymp\bigg(\int_1^\infty x^{-2s}dx\bigg)^n=(1-2s)^{-n},\label{eq:fs}
	\end{align}
	which clearly implies the first point of the lemma.

	By H\"older inequality, $\log f(s)$ is convex on $s>1/2$. Hence, $\log f(s)$ is continuous on $s>1/2$, so is $f(s)$. By \eqref{eq:fs}, it is easily seen that $f(s)$ goes to infinity as $s\downarrow 1/2$.

	Items (1) and (2) follows from the continuity of $f(s)$ and the definitions of $s_{n,i}$ ($i=1,2,3$) directly.
\end{proof}

%One can not infer that $s_{n,2}> 1/2$ because, recall its definition,  $a_1(z_n)^{1-s}/B^{ns}$ may be  large than $1$.

\begin{lem}\label{l:a1(zn)>}
Let $s_{n,i}$ be as in \eqref{eq:sni}, $i=1,2,3$. The following statements hold:
\begin{enumerate}
	\item If $s_{n,1}\le s_{n,2}$, then $a_1(z_n)\ge B^{ns_{n,1}}$.
	\item If $s_{n,1}>s_{n,2}$, then $a_1(z_n)<B^{ns_{n,2}}$.
	\item If $s_{n,2}<s_{n,3}$, then $a_1(z_n)<B^{ns_{n,3}/2}$.
		\item If $s_{n,2}\ge s_{n,3}$, then $a_1(z_n)\ge B^{ns_{n,2}/2}$.
\end{enumerate}
\end{lem}
\begin{proof}
	(1) Since $s_{n,1}\le s_{n,2}$, by Lemma \ref{l:>1/2} (2) and the definitions of $s_{n,1}$ and $s_{n,2}$,
	\[\begin{split}
		&\sum_{a_1,\dots,a_n\in\N}\frac{1}{q_n(a_1,\dots,a_n)^{2s_{n,1}}B^{ns_{n,1}^2}}
		= 1\le \sum_{a_1,\dots,a_n\in\N}\frac{a_1(z_n)^{1-s_{n,1}}}{q_n(a_1,\dots,a_n)^{2s_{n,1}}B^{ns_{n,1}}},
	\end{split}\]
	which is equivalent to
	\[a_1(z_n)\ge B^{ns_{n,1}}.\]

The proofs of the items (2)--(4) follow the same lines as item (1).
\end{proof}
\section{Upper bound of $\hdim E_2(\{z_n\}_{n\ge 1},B)$}
In the remaining of this paper, to simplify the notation, sequences of natural numbers will be denoted by letters in boldface: $\bm a,\bm b,\dots$.

Let
\[
F_n:=\bigcup_{\bm a\in\N^n}\{x\in I_n(\bm a):|T^nx-z_n||T^{n+1}x-Tz_n|<B^{-n}\}.
\]
It follows that
\[E_2(\{z_n\}_{n\ge 1},B)=\bigcap_{N=1}^\infty\bigcup_{n=N}^\infty F_n.\]
Our next objective is to find a suitable cover of $F_n$ for $n\ge 1$, which depends on the values of $ a_1(z_n) $. The following identity will be crucial for this purpose. Note that for any $x=[a_1(x),\dots,a_n(x),\dots]$,
\[T^nx=[a_{n+1}(x),\dots]=\frac{1}{a_{n+1}(x)+T^{n+1}x}.\]
Then, it follows that for $z_n\ne0$,
\begin{align}
	|T^nx-z_n|=&\bigg|\frac{1}{a_{n+1}(x)+T^{n+1}x}-\frac{1}{a_1(z_n)+Tz_n}\bigg|\notag\\
	=&\bigg|\frac{(a_1(z_n)-a_{n+1}(x))+(Tz_n-T^{n+1}x)}{(a_{n+1}(x)+T^{n+1}x)(a_1(z_n)+Tz_n)}\bigg|.\label{eq:identity}
\end{align}
The proof of the upper bound relies on investigating the finer structure of $F_n$, which will be divided into two cases presented in the following two subsections respectively.

Let $s>\limsup\limits_{n \to \infty} s_n$, where $s_n$ is given in \eqref{eq:sn}. Then, there exists an $\epsilon>0$ for which
\begin{equation}\label{eq:sandepsilon}
	s-\epsilon>s_n\quad\text{for all large $n$}.
\end{equation}
\subsection{Optimal cover of $F_n$ when $n$ is large enough and $s_{n,1}\le s_{n,2}$}
Suppose that $n$ is large enough so that \eqref{eq:sandepsilon} is satisfied. Since {\color{red}$s_{n,1}\le s_{n,2}$}, we have
\begin{equation}\label{eq:sn<}
	s_n=s_{n,1}<s-\epsilon
\end{equation}
 with $s$ and $\epsilon$ given in \eqref{eq:sandepsilon}, and by Lemma \ref{l:a1(zn)>}
\[a_1(z_n)\ge B^{ns_{n,1}}.\]

Now, consider the intersection of $F_n$ with $(n+1)$th level cylinders. For any $ (\bm a,a_{n+1})\in\N^{n+1} $ with $\bm a\in\N^n$ and $ a_{n+1}\le B^{ns_{n,1}}/2$, let
\[J_{n+1}(\bm a,a_{n+1}):=F_n\cap I_{n+1}(\bm a,a_{n+1}).\]

Let $x\in I_{n+1}(\bm a,a_{n+1})$. Then, we have $a_{n+1}(x)=a_{n+1}$, and so $|a_{n+1}(x)-a_1(z_n)|\ge B^{ns_{n,1}}/2$, which is much larger than $|Tz_n-T^{n+1}x|$. Applying the identity \eqref{eq:identity} and using $0\le T^{n+1}x, Tz_n\le 1$, we get
\begin{equation}\label{eq:J1}
	\frac{|a_1(z_n)-a_{n+1}|}{8a_1(z_n)a_{n+1}}\le |T^nx-z_n|\le \frac{2|a_1(z_n)-a_{n+1}|}{a_1(z_n)a_{n+1}}.
\end{equation}
If $x$ also belongs to $J_{n+1}(\bm a,a_{n+1})$, then it satisfies the above inequality and $ |T^nx-z_n||T^{n+1}x-Tz_n|\le B^{-n} $.
Thus,
\[ |T^{n+1}x-Tz_n|\le \frac{8a_1(z_n)a_{n+1}}{|a_1(z_n)-a_{n+1}|B^n},\]
which implies that
\[\begin{split}
	J_{n+1}(\bm a,a_{n+1})
	\subset&\bigg\{x\in I_{n+1}(\bm a,a_{n+1}):|T^{n+1}x-Tz_n|\le \frac{8a_1(z_n)a_{n+1}}{|a_1(z_n)-a_{n+1}|B^n}\bigg\}\\
	=&\bigg\{x\in I_{n+1}(\bm a,a_{n+1}):T^{n+1}x\in B\bigg(Tz_n,\frac{8a_1(z_n)a_{n+1}}{|a_1(z_n)-a_{n+1}|B^n}\bigg)\bigg\}.
\end{split}\]
Since for any $x\in I_{n+1}(\bm a,a_{n+1})$,
\begin{equation}\label{eq:T'}
	q_{n+1}(\bm a,a_{n+1})^2\le (T^{n+1})'(x)\le2q_{n+1}(\bm a,a_{n+1})^2,
\end{equation}
it follows that
\begin{align}
	|J_{n+1}(\bm a,a_{n+1})|&\le \frac{16 a_1(z_n)a_{n+1}}{q_{n+1}(\bm a,a_{n+1})^2|a_1(z_n)-a_{n+1}|B^n}\notag\\
	&\le\frac{16 a_1(z_n)}{q_{n}(\bm a)^2a_{n+1}|a_1(z_n)-a_{n+1}|B^n}.\label{eq:J1<}
\end{align}

We take this opportunity to infer that $J_{n+1}(\bm a,a_{n+1})$ contains an interval of length comparable to \eqref{eq:J1<}, which will be needed in the subsequent proof of the lower bound of $\hdim E_2(\{z_n\}_{n\ge1}, B)$. Note that \eqref{eq:J1} holds for any $x \in I_{n+1}(\bm a,a_{n+1})$. If $x$ further satisfies
\[ |T^{n+1}x-Tz_n| \le \frac{a_1(z_n)a_{n+1}}{2|a_1(z_n)-a_{n+1}|B^n},\]
then by the second inequality in \eqref{eq:J1},  one has $ |T^nx-z_n||T^{n+1}x-Tz_n|\le B^{-n} $,
which implies
\[J_{n+1}(\bm a,a_{n+1})\supset \bigg\{x\in I_{n+1}(\bm a,a_{n+1}):T^{n+1}x\in B\bigg(Tz_n,\frac{a_1(z_n)a_{n+1}}{2|a_1(z_n)-a_{n+1}|B^n}\bigg)\bigg\}.\]
By \eqref{eq:T'}, $J_{n+1}(\bm a,a_{n+1})$ contains an interval with length greater than
\begin{equation}\label{eq:J1>}
	\frac{a_1(z_n)a_{n+1}}{4q_{n+1}(\bm a,a_{n+1})^2|a_1(z_n)-a_{n+1}|B^n}\ge \frac{ a_1(z_n)}{16q_{n}(\bm a)^2 a_{n+1}|a_1(z_n)-a_{n+1}|B^n}.
\end{equation}

Note that $F_n$ can be covered by the union of the interval $ \bigcup_{a_{n+1}>B^{ns_{n,1}}/2}I_{n+1}(\bm a,a_{n+1}) $ and the sets $  J_{n+1}(\bm a,a_{n+1})$ with $a_{n+1}\le B^{ns_{n,1}}/2$.
Therefore, by the previous discussion the $s$-volume of optimal cover of $F_n$ can be estimated as follows
\[\begin{split}
	&\sum_{\bm a\in\N^n}\bigg(\frac{1}{q_n(\bm a)^{2s}(B^{ns_{n,1}}/2)^s}+\sum_{a_{n+1}\le B^{ns_{n,1}}/2 }\frac{16^{s}a_1(z_n)^{s}}{q_{n}(\bm a)^{2s}a_{n+1}^{s}|a_1(z_n)-a_{n+1}|^{s}B^{ns}}\bigg)\\
	=&\sum_{\bm a\in\N^n}\bigg(\frac{1}{q_n(\bm a)^{2s}(B^{ns_{n,1}}/2)^s}+\frac{16^s}{q_n(\bm a)^{2s}B^{ns}}\cdot\sum_{a_{n+1}\le B^{ns_{n,1}}/2 }\frac{a_1(z_n)^{s}}{a_{n+1}^{s}|a_1(z_n)-a_{n+1}|^{s}}\bigg)\\
	\asymp&\sum_{\bm a\in\N^n}\bigg(\frac{B^{-n s s_{n,1}}}{q_n(\bm a)^{2s}}+\frac{B^{ns_{n,1}(1-s)-ns}}{q_{n}(\bm a)^{2s}}\bigg)
	\asymp\sum_{\bm a\in\N^n}\frac{B^{-n s s_{n,1}}}{q_n(\bm a)^{2s}},
\end{split}\]
where we use the main ideas of the proof of Lemma \ref{l:summation2} with $t=s>s_{n,1}+\epsilon>1/2$ (by Lemma \ref{l:>1/2} (1)), $a=a_1(z_n)$ and $b=a_{n+1}$ for the inner-most summation in the second formula. By the definition of $s_{n,1}$ and the fact that $s>s_{n,1}+\epsilon$, we see that
\begin{equation}\label{eq:sum1<}
	\sum_{\bm a\in\N^n}\frac{B^{-n s s_{n,1}}}{q_n(\bm a)^{2s}}\le 2^{-n\delta_1}
\end{equation}
for some $\delta_1>0$ depending on $s$ only.

\subsection{Optimal cover of $F_n$ when $n$ is large enough and $s_{n,1} > s_{n,2}$} Suppose that $n$ is large enough so that \eqref{eq:sandepsilon} is satisfied. Since $s_{n,1}> s_{n,2}$, we have
\begin{equation}\label{eq:sn<2}
	s_n=\max \{ s_{n,2},s_{n,3} \}<s-\epsilon
\end{equation}
with $s$ and $\epsilon$ given in \eqref{eq:sandepsilon}, and by Lemma \ref{l:a1(zn)>}
\[a_1(z_n)< B^{ns_{n,2}}.\]
Follow the same notation in the last section, still let
$J_{n+1}(\bm a,a_{n+1}):=F_n\cap I_{n+1}(\bm a,a_{n+1})$ with $ (\bm a,a_{n+1})\in\N^{n+1} $. Let $x\in I_{n+1}(\bm a,a_{n+1})$, and so $a_{n+1}(x)=a_{n+1}$. The discussion is split into three subcases.

\noindent\textbf{Subcase (1A)}: $ |a_1(z_n)-a_{n+1}|>1 $. Since $|Tz_n-T^{n+1}x|<1$, applying the identity \eqref{eq:identity}, we can get
\[\frac{|a_1(z_n)-a_{n+1}|}{8a_1(z_n)a_{n+1}}\le |T^nx-z_n|\le \frac{2|a_1(z_n)-a_{n+1}|}{a_1(z_n)a_{n+1}}.\]
Since $x\in I_{n+1}(\bm a,a_{n+1})$ is arbitrary, by the same reason as \eqref{eq:J1<} and \eqref{eq:J1>}, we get that $J_{n+1}(\bm a,a_{n+1})$ is contained in an interval with length at most
\begin{equation}\label{eq:Jn3<1}
	\frac{16 a_1(z_n)}{q_{n}(\bm a)^2a_{n+1}|a_1(z_n)-a_{n+1}|B^n}
\end{equation}
and contains an interval with length at least
\begin{equation}\label{eq:Jn3>1}
\frac{a_1(z_n)}{16q_{n}(\bm a)^2a_{n+1}|a_1(z_n)-a_{n+1}|B^n}.
\end{equation}

\noindent\textbf{Subcase (1B)}: $ |a_1(z_n)-a_{n+1}|=1 $.
Applying the identity \eqref{eq:identity} again, we obtain

\[|T^nx-z_n|=\bigg|\frac{1-(Tz_n-T^{n+1}x)}{(a_{n+1}+T^{n+1}x)(a_1(z_n)+Tz_n)}\bigg|\ge \frac{|1-(Tz_n-T^{n+1}x)|}{4a_1(z_n)a_{n+1}}.\]
Denote
\[J_{n+1}^{(1)}(\bm a,a_{n+1}):=\{x\in I_{n+1}(\bm a,a_{n+1}):|T^{n+1}x-Tz_n|\le 8a_1(z_n)a_{n+1}B^{-n}\}\]
and
\[J_{n+1}^{(2)}(\bm a,a_{n+1}):=\{x\in I_{n+1}(\bm a,a_{n+1}):|T^{n+1}x-Tz_n|\ge 1-8a_1(z_n)a_{n+1}B^{-n}\}.\]
If  $ 8a_1(z_n)a_{n+1} B^{-n} \ge 1/2 $, then the union of the above two sets is $ I_{n+1}(\bm a,a_{n+1}) $, which can obviously cover $J_{n+1}(\bm a,a_{n+1})$.

Now suppose that $ 8a_1(z_n)a_{n+1} B^{-n}< 1/2$. Let $y \notin J_{n+1}^{(1)}(\bm a,a_{n+1})\cup J_{n+1}^{(2)}(\bm a,a_{n+1})$. There are two cases, one is $ 1/2\le |T^{n+1}y-Tz_n|<1-8a_1(z_n)a_{n+1} B^{-n} $, and the other is $ 8a_1(z_n)a_{n+1} B^{-n}<|T^{n+1}y-Tz_n|<1/2 $.
For the first case,
\[\begin{split}
	|T^ny-z_n||T^{n+1}y-Tz_n|&\geq \frac{|1-(Tz_n-T^{n+1}y)|}{4a_1(z_n)a_{n+1}}\cdot|T^{n+1}y-Tz_n|\\
	&> \frac{8a_1(z_n)a_{n+1} B^{-n}}{4a_1(z_n)a_{n+1}}\cdot\frac{1}{2}=B^{-n}.
\end{split}\]
By employing the same strategy on the other case, one has for $y\notin J_{n+1}^{(1)}(\bm a,a_{n+1})\cup J_{n+1}^{(2)}(\bm a,a_{n+1})$,
\[|T^ny-z_n||T^{n+1}y-Tz_n|>B^{-n},\]
which implies that $y\notin J_{n+1}(\bm a,a_{n+1})$.
{\color{red}Summarizing},
\[J_{n+1}(\bm a,a_{n+1})\subset J_{n+1}^{(1)}(\bm a,a_{n+1})\cup J_{n+1}^{(2)}(\bm a,a_{n+1}),\]
and therefore by the same reason as \eqref{eq:J1<}, $J_{n+1}(\bm a,a_{n+1})$ can be covered by two intervals whose length is at most
\begin{equation}\label{eq:Jn3<2}
	\frac{16a_1(z_n)B^{-n}}{q_{n}(\bm a)^2a_{n+1}|a_1(z_n)-a_{n+1}|}.
\end{equation}

\noindent\textbf{Subcase (1C)}: $ a_1(z_n)=a_{n+1} $. Using the identity \eqref{eq:identity},  one has
\[\frac{|T^{n+1}x-Tz_n|^2}{4a_1(z_n)^2}\le |T^nx-z_n||T^{n+1}x-Tz_n|\le \frac{|T^{n+1}x-Tz_n|^2}{a_1(z_n)^2}.\]
Clearly, the following holds
\[\begin{split}
	&\{x\in I_{n+1}(\bm a,a_{n+1}):|T^{n+1}x-Tz_n| \le a_1(z_n)B^{-n/2}\}\\
	\subset  &J_{n+1}(\bm a,a_{n+1}) \\
	\subset  &\{x\in I_{n+1}(\bm a,a_{n+1}):|T^{n+1}x-Tz_n|\le 2 a_1(z_n)B^{-n/2}\}.
\end{split}\]
This together with \eqref{eq:T'} gives
\begin{equation}\label{eq:Jn3<3}
	|J_{n+1}(\bm a,a_{n+1})|\le \frac{2a_1(z_n) B^{-n/2}}{q_{n+1}(\bm a,a_1(z_n))^2}\le\frac{2 B^{-n/2}}{q_n(\bm a)^2a_1(z_n)}
\end{equation}
and $J_{n+1}(\bm a,a_{n+1})$ contains an interval with length greater than
\begin{equation}\label{eq:Jn3>3}
\frac{a_1(z_n) B^{-n/2}}{q_{n+1}(\bm a,a_1(z_n))^2}\ge\frac{B^{-n/2}}{4q_n(\bm a)^2a_1(z_n)}.
\end{equation}

Combining the estimations \eqref{eq:Jn3<1}, \eqref{eq:Jn3<2} and \eqref{eq:Jn3<3}, the $s$-volume of optimal cover of $F_n$ can be estimated as follows
\[\begin{split}
	\ll&\sum_{\bm a\in\N^n}\bigg(\sum_{a_{n+1}\ne a_1(z_n)}\frac{a_1(z_n)^s B^{-ns}}{q_{n}(\bm a)^{2s}a_{n+1}^s|a_1(z_n)-a_{n+1}|^s}+\frac{B^{-ns/2}}{q_n(\bm a)^{2s}a_1(z_n)^{s}}\bigg)\\
	\asymp&\sum_{\bm a\in\N^n}\bigg(\frac{a_1(z_n)^{1-s} B^{-ns}}{q_{n}(\bm a)^{2s}}+\frac{B^{-ns/2}}{q_n(\bm a)^{2s}a_1(z_n)^{s}}\bigg),
\end{split}\]
where we have used Lemma \ref{l:summation2} with $t=s>\max \{ s_{n,2},s_{n,3} \}+\epsilon >1/2$ (by Lemma \ref{l:>1/2} (1)), $a=a_1(z_n)$ and $b=a_{n+1}$ for the inner-most summation.
By the definitions of $s_{n,2}$ and $s_{n,3}$ and $s>\max \{ s_{n,2},s_{n,3} \}+\epsilon$ (see \eqref{eq:sn<2}), we have
\begin{equation}\label{eq:sum2<}
	\sum_{\bm a\in\N^n}\bigg(\frac{a_1(z_n)^{1-s} B^{-ns}}{q_{n}(\bm a)^{2s}}+\frac{B^{-ns/2}}{q_n(\bm a)^{2s}a_1(z_n)^{s}}\bigg)\le 2^{-n\delta_2}
\end{equation}
for some $\delta_2>0$ depending on $s$ only.
\subsection{Completing the proof of the upper bound of $\hdim E_2(\{z_n\}_{n\ge 1},B)$} By the previous two subsections, we see that for given $s>\limsup\limits_{n \to \infty} s_n$, by \eqref{eq:sum1<} and \eqref{eq:sum2<} there exists a cover of $F_n$ for which the corresponding $s$-volume
\[\ll 2^{-n\min \{ \delta_1,\delta_2 \}},\]
where $\delta_1>0$ and $\delta_2>0$ are respectively given in \eqref{eq:sum1<} and \eqref{eq:sum2<} and depend on $s$ only.
Note that for any $N\in\N$,
\[E_2(\{z_n\}_{n\ge 1}, B)\subset \bigcup_{n=N}F_n.\]
By the definition of Hausdorff measure, one has
\[\mathcal H^s(E_2(\{z_n\}_{n\ge 1}, B))\ll\lim_{N\to\infty}\sum_{n=N}^{\infty}2^{-n\min \{ \delta_1,\delta_2 \}}=0,\]
which implies that $ \hdim E_2(\{z_n\}_{n\ge 1}, B)\le s $.
By the arbitrariness of $ s $ ($>\limsup\limits_{n \to \infty} s_n$), we have
\[\hdim E_2(\{z_n\}_{n\ge 1}, B)\le \limsup_{n\to\infty} s_n,\]
which is what we want.
%By the definition of $s$-dimensional Hausdorff measure,
%\[\mathcal H^t(E(\{z_n\}_{n\ge 1},B))\le\lim_{N\to\infty}\sum_{n=N}^{\infty}2^{-\epsilon sn}=0.\]
%Therefore, by the arbitrariness of $\varepsilon$ and $s$, we have
%\[\hdim E(\{z_n\}_{n\ge 1},B)\le \max(s_2,s_3).\]

\section{Lower bound of $\hdim E_2(\{z_n\}_{n\ge 1},B)$}
Before proving the lower bound, we list some definitions  and auxiliary results which will be used later.

For any set $E \subset \R$, its $s$-dimensional Hausdorff content is given by
\[
\mathcal H^s_\infty (E):= \inf \Big\{ \sum_{i=1}^{\infty} |B_i|^s : E \subset \bigcup_{i=1}^{\infty} B_i ~\text{where}~ B_i ~\text{are~open~balls} \Big\}.
\]
With this definition, the lower bound of the Hausdorff dimension of a limsup set can be estimated by verifying some Hausdorff content bounds. The following lemma is a variant version of Falconer's sets of large intersection condition. For the details, see \cite[Corollary 2.6]{He2024}.

%\begin{defn}
%Let $0<s \le 1$. We define $ \mathscr G^{s}([0,1]) $ to be the class of $G_\delta$-subsets $F$ of $[0,1]$ such that there exists a constant $c$ with $0< c\le 1$ such that, for any ball $B$,
%\[
%\mathcal H^t_\infty(F\cap B)\ge c \mathcal H^t_\infty(B)
%\]
%for any $0<t<s$.
%\end{defn}

\begin{lem}\label{l:LIP}
	Let $\{B_k\}_{k\ge 1}$ be a sequence of balls in $[0,1]$ that satisfies $\mathcal L(\limsup\limits_{k \to \infty} B_k)=\mathcal L([0,1])=1$. Let $ \{E_n\}_{n\ge 1} $ be a sequence of open sets in $[0,1]$ and $ E=\limsup\limits_{n \to \infty} E_n $. Let $ s>0 $. If for any $ 0<t<s $, there exists a constant $ c_t $ such that
	\begin{equation}\label{c:LIP condition}
		\limsup_{n\to\infty} \mathcal H^t_\infty(E_n\cap B_k)\ge c_t|B_k|
	\end{equation}
	holds for all $ B_k $, then $\hdim E\ge s$.
\end{lem}
\begin{rem}
	In fact, limsup set satisfying \eqref{c:LIP condition} has the so-called large intersection property (see \cite{Falconer1994} or \cite{He2024}), which means that the intersection of the sets with countably many similar copies of itself still has Hausdorff dimension at least $s$. In particular, the Hausdorff dimension of such limsup set is at least $s$. This property is beyond the subject of this paper, so we will not go into detail.
\end{rem}

The Hausdorff content of a Borel set is typically estimated by putting measures or mass distributions on it, following the mass distribution described below.

\begin{prop}[{Mass distribution principle \cite[Lemma 1.2.8]{BP2017}}]
	Let $E$ be a Borel subset of $\R$.
	If $E$ supports a strictly positive Borel measure $\mu$ that satisfies
	\[
	\mu(B(x, r)) \le c r^s,
	\]
	for some constant $0< c< \infty$ and for every ball $B(x, r)$, then $ \mathcal H^s_\infty (E) \ge \mu(E)/c$.
\end{prop}

Since
\[\bigcap_{n=1}^\infty\bigcup_{\bm u\in \N^n}I_n(\bm u)=[0,1]\setminus \Q,\]
the sequence of balls demanded in Lemma \ref{l:LIP} can be taken as the set of all cylinders. We will show that for any $0<t<s^*$ (recall that $s^*=\limsup\limits_{n \to \infty} s_n$), there exists a constant $c_t$ depending on $t$ such that for any $k\ge 1$, and an arbitrary $k$th level cylinder $I_k(\bm u)$ with $\bm u\in \N^k$,
\begin{equation}\label{eq:h(u)>}
	\limsup_{n\to\infty} \mathcal H_\infty^t(F_n\cap I_k(\bm u))\ge c_t|I_k(\bm u)|,
\end{equation}
where recall that
\[F_n=\bigcup_{\bm a\in\N^n} \{x\in I_n(\bm a):|T^nx-z_n||T^{n+1}x-Tz_n|\le B^{-n}\}. \]

In view of mass distribution principle, to establish \eqref{eq:h(u)>}, we will construct a measure $ \mu  $ supported on $ F_n\cap I_k(\bm u) $ and then estimate the $\mu$-measure of arbitrary balls. Here and hereafter, $\bm u$ and $t$ will be fixed. The proof is divided into three cases according to how $s^*$ is attained, presented section by section.
\subsection{Case I: $ s^*=\limsup\limits_{n \to \infty} s_n$ is obtained along a subsequence of $\{s_{n,1}\}_{n\ge 1}$}\label{s:1}
Note that $t<s^*$. There exist infinitely many $n$ such that
\begin{equation}\label{eq:sn1-1}
	s_n=s_{n,1}>t.
\end{equation}
For such $n$, by the definition of $s_n$,
\[s_{n,1}\le s_{n,2},\]
which by Lemma \ref{l:a1(zn)>} (1) gives
\begin{equation}\label{eq:a1zn-1}
	a_1(z_n)\ge B^{ns_{n,1}}.
\end{equation}
%Let $l$ be a large integer so that
%\begin{equation}
%	2^{\frac{ns}{2}\epsilon}<
%\end{equation}

Following the same argument as \cite[Lemma 2.4]{WW2008} with some obvious modification, we have
\begin{equation}\label{pressure1}
s^*=\lim_{n\text{ satisfies \eqref{eq:sn1-1}}\atop n\to\infty}s_{n,1}=\inf\{s>0:P(T,-s\log|T'|-s^2\log B)\le 0\}.
\end{equation}
Fix an $\epsilon<s^*-t$. By Proposition \ref{p:property of s} and Lemma \ref{l:>1/2} (2), there exist integers $ \ell>\max \{\log_B4/\epsilon, 2t/\epsilon+1 \} $ and $ M>0 $ such that the unique positive number $s=s(\ell,M)$ satisfying the equation
\begin{equation}\label{eq:defs-1}
	\sum_{\bm a\in\{1,\dots,M \}^\ell}\frac{1}{q_{\ell}(\bm a)^{2s}B^{\ell s^2}}=1
\end{equation}
is greater than $t+\epsilon$. It should be emphasized that $ \ell>\max \{  \log_B4/\epsilon, 2t/\epsilon+1 \}   $ implies that
\begin{equation}\label{eq:Ble-1}
	B^{\ell\epsilon}\ge 4
\end{equation}
and that for any $\bm a\in \N^\ell$
\begin{equation}\label{eq:ql-1}
	q_\ell(\bm a)^{-2s}\le q_\ell(\bm a)^{-2(t+\epsilon)}\le q_\ell(\bm a)^{-2t}2^{-(\ell-1)\epsilon}\le q_\ell(\bm a)^{-2t}2^{-2t}.
\end{equation}

Choose $n\in \N$ so that $n-k\ge \ell kt/\epsilon$ and \eqref{eq:sn1-1} is satisfied, and write $n-k=m\ell+\ell_0$, where $0\le \ell_0<\ell$. By the choice of $n$, one has $m\ge kt/\epsilon$.

From now on, let $\ell, M $ and $n$ be fixed. Let $1^{\ell_0}$ be the word consisting only of $1$ and of length $\ell_0$. Let $\tilde{\bm u}=(\bm u,1^{\ell_0})\in\N^{k+\ell_0}$ and consider the following auxiliary set defined by $(n+1)$th level cylinders:
\begin{equation}\label{eq:auxI-1}
	\begin{split}
		\{I_{n+1}(\tilde {\bm u},\bm a_1,\dots,\bm a_m,b):&\bm a_i\in\{1,\dots,M\}^\ell, 1\le i\le m,\\
		& B^{nt}\le b\le 2B^{nt} \text{ and $b$ is even}\}.
	\end{split}
\end{equation}
The additional requirement here, that $b$ be even, is that any two cylinders in the above set are well separated (see Lemma \ref{l:pro En-1} below). By \eqref{eq:J1>}, there is an interval, denoted by $\mathcal I_1(\tilde{\bm u},\bm a_1,\dots,\bm a_m,b) \footnote{Given that there are three cases (\S \ref{s:1}--\S\ref{s:3}) to consider, each requiring the construction of subsets and measures, we will use subscripts 1, 2, 3 or superscripts (1), (2), (3) to distinguish and avoid burdening of notation, where there is no risk of ambiguity.}$ such that
\begin{equation}\label{eq:interval-1}
	\mathcal I_1(\tilde{\bm u},\bm a_1,\dots,\bm a_m,b)\subset F_n\cap I_{n+1}(\tilde{\bm u},\bm a_1,\dots,\bm a_m,b).
\end{equation}
It is easy to see that $\mathcal I_1(\tilde{\bm u},\bm a_1,\dots,\bm a_m,b)\subset F_n\cap I_{k+\ell_0}(\tilde{ \bm u})\subset F_n\cap I_k(\bm u)$.
Moreover, the length of this interval can be estimated as
\begin{align}
	|\mathcal I_1(\tilde{\bm u},\bm a_1,\dots,\bm a_m,b) |
& \geq \frac{a_1(z_n)}{16 q_{n}(\tilde{\bm u},\bm a_1,\dots,\bm a_m)^2b|a_1(z_n)-b|B^n} \notag\\
& \geq \frac{1}{16 q_n(\tilde{\bm u},\bm a_1,\dots,\bm a_m)^2 b B^{n}}\notag\\
	& \geq \frac{1}{32 q_n(\tilde{\bm u},\bm a_1,\dots,\bm a_m)^2 B^{n(1+t)}},\label{eq:lowerI}
\end{align}
	where we use $a_1(z_n)\ge B^{ns_{n,1}}>B^{nt}B^{n\epsilon}\ge 4B^{nt}$ (see \eqref{eq:a1zn-1} and \eqref{eq:Ble-1}) and $b \leq 2B^{nt}$ in the second and third inequalities. In what follows, we call $\mathcal I_1(\tilde{\bm u},\bm a_1,\dots,\bm a_m,b)$ in $F_n\cap I_{k+\ell_0}(\tilde{\bm u})$ as {\it fundamental interval}, and the cylinders $I_{n+1}(\tilde {\bm u},\bm a_1,\dots,\bm a_m,b)$ and $I_{k+\ell_0+p\ell}(\tilde{\bm u},\bm a_1,\dots,\bm a_p)$ with $0\le p\le m$ as {\it basic cylinders}.

Define a probability measure $\mu_1$ supported on $F_n\cap I_{k+\ell_0}(\tilde {\bm u})\subset F_n\cap I_k(\bm u)$ as follows:
\begin{equation}\label{eq:mu-1}
	\mu_1=\sum_{\bm a_1\in\{1,\dots,M \}^\ell}\cdots\sum_{\bm a_m\in\{1,\dots,M \}^\ell}\sum_{B^{nt}\le b\le 2 B^{nt}\atop b\text{ is even}}\bigg(\prod_{i=1}^m\frac{1}{q_\ell(\bm a_i)^{2s}B^{\ell s^2}}\bigg)\cdot \frac{2}{B^{nt}}\cdot\lm_{\bm a_1,\dots,\bm a_m,b}^{(1)},
\end{equation}
where
\[\lm_{\bm a_1,\dots,\bm a_m,b}^{(1)}:=\frac{\lm|_{\mathcal I_1(\tilde{\bm u},\bm a_1,\dots,\bm a_m,b)}}{\lm(\mathcal I_1(\tilde{\bm u},\bm a_1,\dots,\bm a_m,b))}\]
denotes the normalized Lebesgue measure on the fundamental interval $\mathcal I_1(\tilde{\bm u},\bm a_1,\dots,\bm a_m,b)$. Roughly speaking, we assign each word $(\bm a_1,\dots,\bm a_m,b)$ a different weight: the weights corresponding to $\bm a_i$ and $b$ are, respectively,
\[\frac{1}{q_\ell(\bm a_i)^{2s}B^{\ell s^2}}\quad\text{and}\quad \frac{2}{B^{nt}}.\]
The definition of $s$ (see \eqref{eq:defs-1}) ensures that $\mu_1$ is a probability measure.

The next two lemmas describe the gap between fundamental intervals defined in \eqref{eq:interval-1} and their $\mu_1$-measures, respectively.
\begin{lem}\label{l:pro En-1}
	Let $ \mathcal I_1=\mathcal I_1(\tilde{\bm u},\bm a_1,\dots,\bm a_m,b) $ and $ \mathcal I_1'=\mathcal I_1(\tilde{\bm u},\bm a_1',\dots,\bm a_m',b')$ be two fundamental intervals defined in \eqref{eq:interval-1}. Then, the following statements hold:
	\begin{enumerate}[\upshape(1)]
		\item If  $ (\bm a_1,\dots,\bm a_{p-1})= (\bm a_1',\dots,\bm a_{p-1}') $ but $ \bm a_p \ne \bm a_p' $ for some $ 1\le p\le m $, then
		\[
		\dist(\mathcal I_1,\mathcal I_1')\ge\frac{|I_{k+\ell_0+p \ell}(\tilde{\bm u},\bm a_1,\dots,\bm a_p)|}{2(M+2)^4}.
%		\max\big\{|I_{k+\ell_0+p \ell}(\tilde{\bm u},\bm a_1,\dots,\bm a_p)|,|q_{k+\ell_0+p \ell}(\tilde{\bm u}, \bm a_1',\dots,\bm a_p')|\big\}.
		\]
		\item If $ (\bm a_1,\dots,\bm a_m)=(\bm a_1',\dots,\bm a_m') $ but $b\ne b'$, then
		\[\dist(\mathcal I_1,\mathcal I_1')\geq \frac{|I_{n+1}(\tilde{\bm u},\bm a_1,\dots,\bm a_m,b)|}{32}.\]
	\end{enumerate}
\end{lem}

\begin{proof}

(1) Bear in mind that $\mathcal I_1$ and $\mathcal I_1'$ are two fundamental intervals contained in the $(n+1)$th level cylinders $I_{n+1}(\tilde{\bm u},\bm a_1,\dots,\bm a_m,b)$ and $I_{n+1}(\tilde{\bm u},\bm a_1',\dots,\bm a_m',b')$, respectively. For further discussion, we write $  (\bm a_1,\dots,\bm a_m,b)=(c_1,c_2,\ldots, c_{m\ell},c_{m\ell+1})$ and $ (\bm a_1',\dots,\bm a_m',b') = (c_1',c_2',\ldots, c_{m\ell}',c_{m\ell+1}') $ for the moment.
Assume that  $1\le i\le m\ell+1$ is the smallest integer such that $c_i  \ne c_i'$. By the assumption in item (1), we have $(p-1) \ell < i \le p \ell \le m\ell$.
Therefore, the distance between $\mathcal I_1$ and $\mathcal I_1'$ is majorized by
\begin{equation}\label{eq:disttwocylinder-1}
	\dist \big(I_{k+\ell_0+i+1}(\tilde{\bm u},c_1,\dots,c_{i-1},c_i,c_{i+1}),I_{k+\ell_0+i+1}(\tilde{\bm u},c_1,\dots,c_{i-1},c_i',c_{i+1}')\big).
\end{equation}
Now, we consider two cases.

\noindent\textbf{Case 1:} $1\le i<m\ell$. Since $i<m\ell$, we have $c_{i+1},c_{i+1}' \leq M$.
% and $I_{k+\ell_0+p\ell} (\tilde{\bm u}, \bm a_1,\dots,\bm a_p) \subset I_{k+\ell_0+i-1} (\tilde{\bm u}, c_1, \ldots, c_{i+1}) $.
 By the distribution of cylinders (see Proposition \ref{C-D}), either $I_{k+\ell_0+i+1}(\tilde{\bm u}, c_1, \ldots, c_{i-1},c_i, M+1)$ or $I_{k+\ell_0+i+1}(\tilde{\bm u}, c_1, \ldots, c_{i-1},c_i', M+1)$ lies between two cylinders listed in \eqref{eq:disttwocylinder-1}. Without loss of {\color{red}generality}, assume that this is satisfied by the former one. Then, by \eqref{eq:disttwocylinder-1} and using Proposition \ref{p:basic property} (2) repeatly, we have
\begin{align}
	\dist(\mathcal I_1,\mathcal I_2)\ge&| I_{k+\ell_0+i+1}(\tilde{\bm u}, c_1, \ldots, c_{i-1},c_i, M+1) | \notag\\
	\geq &\frac{1}{2 q_{k+\ell_0+i+1} (\tilde{\bm u}, c_1, \ldots, c_{i-1},c_i, M+1)^2 } \notag\\
	\geq &\frac{1}{2 (M+2)^2 (c_i+1)^2 q_{k+\ell_0+i-1} (\tilde{\bm u}, c_1, \ldots, c_{i-1})^2   } \notag\\
	\geq &\frac{|I_{k+\ell_0+i-1} (\tilde{\bm u}, c_1, \ldots, c_{i-1})|}{2 (M+2)^4}  \notag \\
	\geq &\frac{|I_{k+\ell_0+p\ell} (\tilde{\bm u}, \bm a_1,\dots,\bm a_p)|}{2 (M+2)^4} \label{eq:lowerdist-1}
\end{align}
as desired.

\noindent\textbf{Case 2:} $i=m\ell$. In this case, $c_{i+1}=b$ and $c_{i+1}'=b'$. By the distribution of cylinder (see Proposition \ref{C-D}), we see that either $I_{n+1}(\tilde{\bm u},c_1,\dots,c_{m\ell-1},c_{m\ell},1)$ or $I_{n+1}(\tilde{\bm u},c_1,\dots,c_{m\ell-1},c_{m\ell}',1)$ lies between two cylinders listed in \eqref{eq:disttwocylinder-1}. By the same reason as \eqref{eq:lowerdist-1}, we can obtain the same conclusion.

(2) Without loss of {\color{red}generality}, assume that $b < b' $. Note that by \eqref{eq:auxI-1}, both $b$ and $b'$ are even. Then, for any $b''$ with $b<b''<b'$, the $(n+1)$th level cylinder $I_{n+1}(\tilde{\bm u},\bm a_1,\dots,\bm a_m,b'')$ lies between $\mathcal I_1$ and $\mathcal I_1'$.  This means the distance between $\mathcal I_1$ and $\mathcal I_1'$ is at least
\[\begin{split}
|I_{n+1}(\tilde{\bm u},\bm a_1,\dots,\bm a_m,b'')|&\ge \frac{1}{2q_{n+1}(\tilde{\bm u},\bm a_1,\dots,\bm a_m,b'')^2}\\
&\geq \frac{1}{2 \cdot 4^2 q_{n}(\tilde{\bm u},\bm a_1,\dots,\bm a_m)^2B^{2nt}}\\
& \geq \frac{1}{2 \cdot 4^2 q_{n}(\tilde{\bm u},\bm a_1,\dots,\bm a_m)^2 b^{2}}\\
&\ge \frac{1}{32}|I_{n+1}(\tilde{\bm u},\bm a_1,\dots,\bm a_m,b)|,
\end{split}\]
where we use $B^{nt}\le b'',b\le 2B^{nt}$ in the second and third inequalities.
\end{proof}

\begin{lem}\label{l:meamu}
	Let $ \mu_1$ be as in \eqref{eq:mu-1}. Then, the following statements hold:
	\begin{enumerate}[\upshape(1)]
		\item For any $ (\tilde{\bm u},\bm a_1,\dots,\bm a_p) $ with $ 0 \le p\le m $,
		\[\mu_1 (I_{k+\ell_0+p\ell}(\tilde{\bm u},\bm a_1,\dots,\bm a_p))\le\frac{q_{k+\ell_0+p \ell}(\tilde {\bm u},\bm a_1,\dots,\bm a_p)^{-2t}}{|I_{k+\ell_0}(\tilde{\bm u})|^t}.\]
%		In particular, $ \mu (I_{k+\ell_0}(\tilde{\bm u}))\le\frac{4|I_{k+\ell_0}(\tilde{\bm u})|^{t}}{|I_{k+\ell_0}(\tilde{\bm u})|^t}$ also holds.
		\item For any $ (\tilde{\bm u},\bm a_1,\dots,\bm a_m,b) $,
		\[
		\mu_1({\mathcal I_1}(\tilde{\bm u},\bm a_1,\dots,\bm a_m,b))\le \frac{64 |\mathcal I_1(\tilde{\bm u},\bm a_1,\dots,\bm a_m,b) |^t}{|I_{k+\ell_0}(\tilde {\bm u})|^t}.
		\]
	\end{enumerate}
\end{lem}
\begin{proof}
	(1) The conclusion is clear if $p=0$. Now suppose that $p\ne 0$. By the definition of $ \mu_1 $, one has
	\begin{align}
		&\mu_1(I_{k+\ell_0+p\ell}(\tilde{\bm u},\bm a_1,\dots,\bm a_p))\notag\\
		=&\prod_{i=1}^p\frac{1}{q_\ell(\bm a_i)^{2s}B^{\ell s^2}}\le \prod_{i=1}^p\frac{1}{q_\ell(\bm a_i)^{2s}}\overset{\eqref{eq:ql-1}}{\le} \prod_{i=1}^p\frac{1}{(2q_\ell(\bm a_i))^{2t}}\notag\\
		 \le& \frac{1}{2^{2t} q_{p\ell}(\bm a_1,\dots,\bm a_p)^{2t}}= \frac{q_{k+\ell_0}(\tilde{ \bm u})^{2t}}{q_{k+\ell_0}(\tilde {\bm u})^{2t}}\cdot\frac{1}{2^{2t} q_{p\ell}(\bm a_1,\dots,\bm a_p)^{2t}}\notag\\
		\le&\frac{q_{k+\ell_0}(\tilde {\bm u})^{2t}}{q_{k+\ell_0+p \ell}(\tilde {\bm u},\bm a_1,\dots,\bm a_p)^{2t}}\notag \\
	\le&\frac{q_{k+\ell_0+p \ell}(\tilde {\bm u},\bm a_1,\dots,\bm a_p)^{-2t}}{|I_{k+\ell_0}(\tilde {\bm u})|^t}\label{eq:qkl-1}.
	\end{align}

	(2) Recall that $n=k+\ell_0+m\ell$. It should be noticed that \eqref{eq:qkl-1} is actually an upper bound of $\prod_{i=1}^pq_\ell(\bm a_i)^{-2s}$. In the spirit of \eqref{eq:qkl-1}, by the definition of $ \mu_1$,
	\begin{align}
		\mu_1({\mathcal I_1}(\tilde{\bm u},\bm a_1,\dots,\bm a_m,b))
		=&\bigg(\prod_{i=1}^m\frac{1}{q_\ell(\bm a_i)^{2s}B^{\ell s^2}}\bigg)\cdot \frac{2}{B^{nt}}\notag\\
		\le &\frac{q_{k+\ell_0+m \ell}(\tilde {\bm u},\bm a_1,\dots,\bm a_m)^{-2t}}{|I_{k+\ell_0}(\tilde {\bm u})|^t}\cdot \frac{1}{B^{m\ell s^2} }\cdot \frac{2}{B^{nt}}.\notag
	\end{align}
	Since $m\ge kt/\epsilon$ (which follows from the choice of $n$), one has
	\begin{equation}\label{eq:ml>n-1}
		\begin{split}
			m\ell s^2&\ge m\ell(t+\epsilon)^2 \ge m\ell t^2+2m\ell t\epsilon\ge  m\ell t^2+2 \ell kt^2\\
			&\ge m\ell t^2+(\ell+k)t^2\ge nt^2.
		\end{split}
	\end{equation}
  Therefore,
\begin{align}
	\mu_1({\mathcal I_1}(\tilde{\bm u},\bm a_1,\dots,\bm a_m,b)) & \le \frac{q_{k+\ell_0+m \ell}(\tilde {\bm u},\bm a_1,\dots,\bm a_m)^{-2t}}{|I_{k+\ell_0}(\tilde {\bm u})|^t}\cdot \frac{1}{B^{n t^2} }\cdot \frac{2}{B^{nt}}\notag\\
	&=\frac{1}{|I_{k+\ell_0}(\tilde {\bm u})|^t}\cdot\frac{2}{q_{n}(\tilde {\bm u},\bm a_1,\dots,\bm a_m)^{2t}B^{nt(1+t)}}\label{eq:b<-1}\\
	& \overset{\eqref{eq:lowerI}}{\le} \frac{64 |\mathcal I_1(\tilde{\bm u},\bm a_1,\dots,\bm a_m,b) |^t}{|I_{k+\ell_0}(\tilde {\bm u})|^t}\notag.\qedhere
\end{align}
\end{proof}

\begin{lem}\label{l:holder1}
	Let $ \mu_1$ be as in \eqref{eq:mu-1}. For any $ r>0 $ and $ x\in[0,1] $, we have
	\[\mu_1(B(x,r))\le \frac{16(M+2)^4(M+1)^{2\ell}r^t}{|I_{k+\ell_0}(\tilde {\bm u})|^t}.\]
\end{lem}
\begin{proof}
	Without loss of generality, assume that $ x\in{\mathcal I_1}(\tilde{\bm u},\bm a_1,\dots,\bm a_m,b) $. Obviously, if $ r\ge|I_{k+\ell_0}(\tilde{\bm u})| $, then
	\[\mu_1 (B(x,r))=1\le \frac{r^{t}}{|I_{k+\ell_0}(\tilde{\bm u})|^{t}}.\]
	Hence, it is sufficient to focus on the case $ r<|I_{k+\ell_0}(\tilde{\bm u})| $. Lemma \ref{l:pro En-1} suggests that we need to consider three cases.

	\noindent \textbf{Case 1}: $r\le|I_{n+1}(\tilde{\bm u},\bm a_1,\dots,\bm a_m,b)|/32$. By Lemma \ref{l:pro En-1}, we see that the distance between ${\mathcal I_1}(\tilde{\bm u},\bm a_1,\dots,\bm a_m,b)$ and other fundamental intervals contained in $F_n\cap I_{k+\ell_0}(\tilde{\bm u})$ is at least $r$. So $B(x,r)$ only intersects the fundamental interval ${\mathcal I_1}(\tilde{\bm u},\bm a_1,\dots,\bm a_m,b)$ to which $x$ belongs. It follows that
	\[\begin{split}
		\mu_1(B(x,r))&=\bigg(\prod_{i=1}^m\frac{1}{q_\ell(\bm a_i)^{2s}B^{\ell s^2}}\bigg)\cdot \frac{2}{B^{nt}}\cdot \lm_{\bm a_1,\dots,\bm a_m,b}^{(1)}(B(x,r))\\
		&=\mu_1({\mathcal I_1}(\tilde{\bm u},\bm a_1,\dots,\bm a_m,b))\cdot \lm_{\bm a_1,\dots,\bm a_m,b}^{(1)}(B(x,r)).
	\end{split}\]
	Since
	\[\begin{split}
		\lm_{\bm a_1,\dots,\bm a_m,b}^{(1)}(B(x,r))&\le\min\biggl\{1,\frac{2r}{|{\mathcal I_1}(\tilde{\bm u},\bm a_1,\dots,\bm a_m,b)|}\biggr\}\le \frac{2r^t}{|{\mathcal I_1}(\tilde{\bm u},\bm a_1,\dots,\bm a_m,b)|^t},
	\end{split}\]
	where the last inequality follows from $\min\{a, c\} \leq a^t c ^{1-t}$ for any $t \in [0,1]$.
	 By Lemma \ref{l:meamu} (2), we have
	\[\begin{split}
		\mu_1(B(x,r))
		\leq  & \frac{64|{\mathcal I_1}(\tilde{\bm u},\bm a_1,\dots,\bm a_m,b)|^t}{|I_{k+\ell_0}(\tilde {\bm u})|^t}\cdot \frac{2r^t}{|{\mathcal I_1}(\tilde{\bm u},\bm a_1,\dots,\bm a_m,b)|^t}\\
		= & \frac{128r^t}{|I_{k+\ell_0}(\tilde {\bm u})|^t}.
	\end{split}\]

	\noindent \textbf{Case 2}: $|I_{n+1}(\tilde{\bm u},\bm a_1,\dots,\bm a_m,b)|/32< r\le|I_{n}(\tilde{\bm u},\bm a_1,\dots,\bm a_m)|/(2(M+2)^4)$. In this case, the ball $B(x,r)$ intersects exactly the $n$th level basic cylinder  $I_{n}(\tilde{\bm u},\bm a_1,\dots,\bm a_m)$, but may intersect multiple $(n+1)$th level basic cylinders inside it. Therefore,
	by the definition of $\mu_1$ and \eqref{eq:b<-1},
	\begin{align}
		\mu_1(B(x,r))
		\leq& \#\Delta_1(x)\cdot\max_{B^{nt}\le b\le 2B^{nt}\atop b\text{ is even}} \mu_1({\mathcal I_1}(\tilde{\bm u},\bm a_1,\dots,\bm a_m,b))\notag\\
		\le &\#\Delta_1(x)\cdot \frac{1}{|I_{k+\ell_0}(\tilde {\bm u})|^t}\cdot\frac{2}{q_n(\tilde {\bm u},\bm a_1,\dots,\bm a_m)^{2t}B^{nt(1+t)}},\label{eq:mu1<Delta1}
	\end{align}
	where
	\[\Delta_1(x)=\{B^{nt}\le b\le 2B^{nt}: \text{$b$ is even and } I_{n+1}(\tilde{\bm u},\bm a_1,\dots,\bm a_m,b)\cap B(x,r)\ne\emptyset\}.\]
	In order to get the best upper bound for $ \mu (B(x,r)) $, we need to use two methods to bound $ \#\Delta_1(x) $ from above. First, there are at most $B^{nt}/2$ choices for $b$ and so
	\begin{equation}\label{eq:Del-1}
		\#\Delta_1(x)\le B^{nt}/2.
	\end{equation}

	On the other hand, each $(n+1)$th level basic cylinder $I_{n+1}(\tilde{\bm u},\bm a_1,\dots,\bm a_m,b)$ is of length at least
	\[2^{-1}q_{n+1}(\tilde{\bm u},\bm a_1,\dots,\bm a_m,b)^{-2}\ge 32^{-1}q_{n}(\tilde{\bm u},\bm a_1,\dots,\bm a_m)^{-2}B^{-2nt},\]
	which means that
	\[\#\Delta_1(x)\le 2r\cdot 32 q_{n}(\tilde{\bm u},\bm a_1,\dots,\bm a_m)^{2}B^{2nt}.\]
	This together with \eqref{eq:Del-1} gives
		\[\begin{split}
		\# 	\Delta_1(x)&\le\min\{B^{nt}/2, 2r\cdot 32 q_n(\tilde{\bm u},\bm a_1,\dots,\bm a_m)^2B^{2nt}\}\\
			&\leq 64 B^{nt(1-t)}\cdot( r\cdot q_n(\tilde{\bm u},\bm a_1,\dots,\bm a_m)^{2}B^{2nt})^t\\
			&=64 B^{nt(1+t)}\cdot r^t\cdot q_n(\tilde{\bm u},\bm a_1,\dots,\bm a_m)^{2t}.
		\end{split}\]
	Substituting this upper bound for $\#\Delta_1(x)$ into \eqref{eq:mu1<Delta1}, we get
	\[\begin{split}
		\mu_1(B(x,r))
		\le &64 B^{nt(1+t)}\cdot r^t\cdot q_n(\tilde{\bm u},\bm a_1,\dots,\bm a_m)^{2t}\cdot \frac{1}{|I_{k+\ell_0}(\tilde {\bm u})|^t}\\
&\hspace{2em}	\cdot	\frac{2}{q_n(\tilde {\bm u},\bm a_1,\dots,\bm a_m)^{2t}B^{nt(1+t)}}\\
		=& 128\frac{r^t}{|I_{k+\ell_0}(\tilde {\bm u})|^t}.
	\end{split}\]

	\noindent \textbf{Case 3}: $|I_{k+\ell_0+(p+1) \ell}(\tilde{\bm u},\bm a_1,\dots,\bm a_{p+1})|/(2(M+2)^4) \le r< | I_{k+\ell_0+p\ell}(\tilde{\bm u},\bm a_1,\dots,\bm a_p)|/(2($ $M+2)^4) $ for some $1\le p\le m-1$.  In this case, the ball $B(x,r)$ only intersects one $(k+\ell_0+p \ell)$th level basic cylinder, i.e. $I_{ k+\ell_0+p \ell}(\tilde{\bm u},\bm a_1,\dots,\bm a_p)$. Hence, by Lemma \ref{l:meamu} (1), we get
	\[\begin{split}
		\mu_1(B(x,r))&\le\mu_1(I_{k+\ell_0+p\ell}(\tilde{ \bm u},\bm a_1,\dots,\bm a_p))  \\
   & \le \frac{q_{k+\ell_0+p \ell}(\tilde{\bm u},\bm a_1,\dots,\bm a_{p})^{-2t}}{| I_{k+\ell_0}(\tilde{\bm u})|^t} \\
   & = \frac{q_{k+\ell_0+p \ell}(\tilde{\bm u},\bm a_1,\dots,\bm a_{p})^{-2t} }{| I_{k+\ell_0}(\tilde{\bm u})|^t}\cdot \frac{q_{\ell} (\bm a_{p+1})^{-2t}}{q_{\ell} (\bm a_{p+1})^{-2t}}\\
   & \le \frac{4q_{k+\ell_0+(p+1) \ell}(\tilde{\bm u},\bm a_1,\dots,\bm a_{p+1})^{-2t}}{| I_{k+\ell_0}(\tilde{\bm u})|^t }\cdot (M+1)^{2\ell}\\
   & \le \frac{ 16(M+2)^4(M+1)^{2\ell} r^t}{| I_{k+\ell_0}(\tilde{\bm u})|^t }.
	\end{split}\]

	Combining the estimation given in Cases 1--3, we arrive at the conclusion.
\end{proof}

\begin{proof}[Completing the proof of Theorem \ref{t:main}]
	Recall that $\tilde{ \bm u}=(\bm u, 1^{\ell_0}) \in\N^{k+\ell_0}$. By Proposition \ref{p:basic property} (1), a simple calculation shows that
	\[\frac{|I_{k+\ell_0}(\tilde{\bm u})|^t}{|I_k(\bm u)|^t}\ge\frac{ q_{k+\ell_0}(\tilde{ \bm u})^{-2t}}{2^{2t}\cdot q_{k}({ \bm u})^{-2t}}\ge \frac{1}{2^{4t}\cdot q_{\ell_0}(1^{\ell_0})^{2t}}\ge \frac{1}{2^{4t}\cdot 2^{t\ell_0}}\ge \frac{1}{2^{\ell+4}}.\]
	Therefore, for any $\bm u\in\N^k$, by Lemma \ref{l:holder1} and mass distribution principle, we have
	\[\begin{split}
		\mathcal H_\infty^t(F_n\cap I_{k}(\bm u))&\ge \mathcal H_\infty^t(F_n\cap I_{k+\ell_0}(\tilde{\bm  u}))\\
		&\ge \frac{1}{16(M+2)^4(M+1)^{2\ell}}  |I_{k+\ell_0}(\tilde{\bm  u})|^t  \mu_1 (F_n\cap I_{k+\ell_0}(\tilde{\bm  u}))\\
		&\ge \frac{1}{2^{\ell+8}(M+2)^4(M+1)^{2\ell}} |I_{k}(\bm u)|^t\\
		&\ge \frac{1}{2^{\ell+8}(M+2)^4(M+1)^{2\ell}} |I_{k}(\bm u)|,
	\end{split}\]
	where the last inequality follows from $t<s^*\le 1$.
	Since $\ell$ and $M$ depend on $t$ only, and since the above Hausdorff content bound hold for infinitely many $n\in\N$, by Lemma \ref{l:LIP} we have
	\[\hdim E_2(\{z_n\}_{n\ge 1},B)\ge s^*.\qedhere\]
\end{proof}
\subsection{Case II: $s^*=\limsup\limits_{n \to \infty} s_n$ is obtained along a subsequence of $\{s_{n,2}\}_{n\ge 1}$}\label{s:2}
In this case, there exist infinitely many $n$ such that
\begin{equation}\label{eq:sn-2}
	s_n=s_{n,2}>t.
\end{equation}
For such $n$, by the definition of $s_n$,
\[s_{n,1}> s_{n,2}\ge s_{n,3},\]
which by Lemma \ref{l:a1(zn)>} (2) and (4) gives
\begin{equation}\label{eq:a1zn<2}
	B^{ns_{n,2}/2}\le a_1(z_n) < B^{ns_{n,2}}.
\end{equation}
By taking a subsequence, assume that integers satisfying \eqref{eq:sn-2} will ensure the existence of the following limit
 \[(s^* \log B)/2\le\lim_{n\text{ satisfies \eqref{eq:sn-2}}\atop n\to\infty}\frac{\log a_1(z_n)}{n}=: \alpha\le s^* \log B.\]
Recall the definition of $s_{n,2}$ and using the continuity of pressure functions, the above discussion gives
\begin{equation}\label{pressure2}
s^*=\inf\{s\in [0,1]:P(T, -s\log|T'|-s\log B+(1-s)\alpha)\le 0\}.
\end{equation}

Since most of the arguments in this subsection are quite identical to the last subsection, we will follow the same notations when there is no risk of ambiguity. Besides, to keep the paper in a managable length, some proofs will not be detailed here if they are similar to those in Section \ref{s:1}. Instead, we will present only the main ideas.

Fix an $\epsilon<s^*-t$ . By Proposition \ref{p:property of s}, there exist integers $ \ell \ge 2t/\epsilon +1$ and $ M $ such that the unique positive number $s=s(\ell,M)$ satisfying the equation
\begin{equation}\label{eq:defs-2}
	\sum_{\bm a\in\{1,\dots,M \}^\ell}\frac{e^{\alpha\ell (1-s)}}{q_\ell(\bm a)^{2s}B^{\ell s}}=1,
\end{equation}
is greater than $t+\epsilon$.

Choose a large integer $n\in \N$ so that \eqref{eq:sn-2} is satisfied and
\begin{equation}\label{eq:condn2}
	n-k\ge \ell,\quad e^{n(\alpha-\epsilon)}\le a_1(z_n)\le e^{n(\alpha+\epsilon)}.
\end{equation}
Write $n-k=m\ell+\ell_0$, where $0\le \ell_0<\ell$.

Analogously, let $\tilde{\bm u}=(\bm u,1^{\ell_0})\in\N^{k+\ell_0}$ and consider the following set defined by $(n+1)$th level cylinders:
\[\begin{split}
	\{I_{n+1}(\tilde {\bm u},\bm a_1,\dots,\bm a_m,c): ~&\bm a_i\in\{1,\dots,M\}^\ell, 1\le i\le m,\\
	& 2e^{(\alpha+\epsilon)n}\le c \le 3e^{(\alpha+\epsilon)n} \text{ and $c$ is even}\}.
\end{split}\]
Similar to Section 4.1, we restrict $c$ to be even for the only reason that any two cylinders in above set are well seperated (see Lemma \ref{l:pro En-2} below). By \eqref{eq:Jn3>1}, there is an interval, denoted by $\mathcal I_2(\tilde{\bm u},\bm a_1,\dots,\bm a_m,c)$ such that
\begin{equation}\label{eq:interval-2}
	\mathcal I_2(\tilde{\bm u},\bm a_1,\dots,\bm a_m,c)\subset F_n\cap I_{n+1}(\tilde{\bm u},\bm a_1,\dots,\bm a_m,c).
\end{equation}
Moreover, the length of this interval can be estimated as
\begin{align}
	|\mathcal I_2(\tilde{\bm u},\bm a_1,\dots,\bm a_m,c)|
	&\ge \frac{a_1(z_n)}{16q_{n}(\tilde{\bm u},\bm a_1,\dots,\bm a_m)^2 c |a_1(z_n)-c|B^n} \notag\\
	& \ge \frac{1}{16\cdot 3 e^{2n\epsilon} q_n(\tilde{\bm u},\bm a_1,\dots,\bm a_m)^2cB^{n}} \label{lowerdist-2},
\end{align}
where the last inequality follows from $c\le 3e^{(\alpha+\epsilon)n}$ and $ e^{n(\alpha-\epsilon)}\le a_1(z_n)\le e^{n(\alpha+\epsilon)} $ (see \eqref{eq:condn2}).
%Following the same terminology in the last section, we still call $\mathcal I_2(\tilde{\bm u},\bm a_1,\dots,\bm a_m, c)$ as {\it fundamental interval}, and the cylinders $I_{n+1}(\tilde {\bm u},\bm a_1,\dots,\bm a_m,c)$ and $I_{k+\ell_0+p\ell}(\tilde{\bm u},\bm a_1,\dots,\bm a_p)$ as {\it basic cylinders}.

Define a probability measure $\mu_2$ supported on $F_n\cap I_{k+\ell_0}(\tilde{\bm u})\subset F_n\cap I_k(\bm u)$ as follows:
\[\mu_2=\sum_{\bm a_1\in\{1,\dots,M \}^\ell}\cdots\sum_{\bm a_m\in\{1,\dots,M \}^\ell}\sum_{2e^{n(\alpha+\epsilon)}\le c\le 3e^{n(\alpha+\epsilon)}\atop c\text{ is even}}\bigg(\prod_{i=1}^m\frac{e^{\ell \alpha(1-s)}}{q_\ell(\bm a_i)^{2s}B^{\ell s}}\bigg)\cdot \frac{2}{e^{n(\alpha+\epsilon)}}\cdot\lm^{(2)}_{\bm a_1,\dots,\bm a_m,c},\]
where
\[\lm^{(2)}_{\bm a_1,\dots,\bm a_m,c}:=\frac{\lm|_{\mathcal I_2(\tilde{\bm u},\bm a_1,\dots,\bm a_m,c)}}{\lm(\mathcal I_2(\tilde{\bm u},\bm a_1,\dots,\bm a_m,c))}.\]

The next two lemmas describe the gap between fundamental intervals defined in \eqref{eq:interval-2} and their $\mu_2$-measures, respectively. The first one follows the same lines as the proof of Lemma \ref{l:pro En-1}.
\begin{lem}\label{l:pro En-2}
	Let $ \mathcal I_2=\mathcal I_2(\tilde{\bm u},\bm a_1,\dots,\bm a_m,c) $ and $ \mathcal I'_2=\mathcal I_2(\tilde{\bm u},\bm a_1',\dots,\bm a_m',c')$ be two intervals defined in \eqref{eq:interval-2}. Then, the following statements hold:
	\begin{enumerate}[\upshape(1)]
		\item If $ (\bm a_1,\dots,\bm a_{p-1})= (\bm a_1',\dots,\bm a_{p-1}') $ but $ \bm a_p \ne \bm a_p' $ for some $ 1\le p\le m $, then
		\[
		\dist(\mathcal I_2,\mathcal I_2')\ge\frac{|I_{k+\ell_0+p \ell}(\tilde{\bm u},\bm a_1,\dots,\bm a_p)|}{2(M+2)^4}
		%		\max\big\{|I_{k+\ell_0+p \ell}(\tilde{\bm u},\bm a_1,\dots,\bm a_p)|,|q_{k+\ell_0+p \ell}(\tilde{\bm u}, \bm a_1',\dots,\bm a_p')|\big\}.
		\]
		\item If $ (\bm a_1,\dots,\bm a_m)=(\bm a_1',\dots,\bm a_m') $ but $c\ne c'$, then
		\[\dist(\mathcal I_2,\mathcal I_2')\geq \frac{1}{18}|I_{n+1}(\tilde{\bm u},\bm a_1,\dots,\bm a_m,c)|.\]
	\end{enumerate}
\end{lem}

Instead of giving complete proof of the following lemma, we merely point out which changes have to be made.
\begin{lem}\label{l:meamu2}
	\begin{enumerate}[\upshape(1)]
		\item For any $ (\tilde{\bm u},\bm a_1,\dots,\bm a_p) $ with $ 0\le p\le m $,
		\[\mu_2 (I_{k+\ell_0+p\ell}(\tilde{\bm u},\bm a_1,\dots,\bm a_p))\ll\frac{q_{k+\ell_0+p \ell}(\tilde{\bm u},\bm a_1,\dots,\bm a_p)^{-2t}}{|I_{k+\ell_0}(\tilde {\bm u})|^{t}}.\]
		\item For any $ (\tilde{\bm u},\bm a_1,\dots,\bm a_m,c) $,
		\[\mu_2(\mathcal I_2(\tilde{\bm u},\bm a_1,\dots,\bm a_m,c))\ll \frac{|\mathcal I_2(\tilde{\bm u},\bm a_1,\dots,\bm a_m,c)|^t}{|I_{k+\ell_0}(\tilde {\bm u})|^t }.\]
	\end{enumerate}
\end{lem}

\begin{proof}[Sketch proof]
	(1) Note that
	\[\mu_2(I_{k+\ell_0+p\ell}(\tilde{\bm u},\bm a_1,\dots,\bm a_p))=\prod_{i=1}^p\frac{e^{\ell \alpha (1-s)} }{q_\ell(\bm a_i)^{2s}B^{\ell s}}.\]
	By \eqref{eq:a1zn<2} and \eqref{eq:condn2},
	\[e^{n\alpha}\le e^{n\epsilon}a_1(z_n)\le e^{n\epsilon}B^{ns_{n,2}}\le e^{n\epsilon}B^{n(s+O(\epsilon))},\]
	which is equivalent to
	\begin{equation}\label{eq:ealpha<}
		e^{\alpha}\le  e^\epsilon B^{s+O(\epsilon)}.
	\end{equation}
	By decreasing $\epsilon$ if necessary, we have
	\[e^{\alpha(1-s)}\le  B^{s}.\]
	Therefore,
	\begin{equation}\label{eq:qqq<-2}
		\mu_2(I_{k+\ell_0+p\ell}(\tilde{\bm u},\bm a_1,\dots,\bm a_p))\le \prod_{i=1}^p\frac{1 }{q_\ell(\bm a_i)^{2s}} \le \frac{q_{k+\ell_0+p \ell}(\tilde{\bm u},\bm a_1,\dots,\bm a_p)^{-2t}}{|I_{k+\ell_0}(\tilde {\bm u})|^{t}},
	\end{equation}
	where the last inequality follows the same argument identical to \eqref{eq:qkl-1}.

(2) By the definition of $\mu_2$,
\begin{align*}
	\mu_2(\mathcal I_2(\tilde{\bm u},\bm a_1,\dots,\bm a_m,c))
	=&\bigg(\prod_{i=1}^m\frac{e^{ \ell \alpha (1-s)}}{q_\ell(\bm a_i)^{2s}B^{\ell s}}\bigg)\cdot \frac{2}{e^{n(\alpha+\epsilon)}}\notag\\
	=& \bigg(\prod_{i=1}^m\frac{1}{q_\ell(\bm a_i)^{2s} }\bigg) \cdot \frac{e^{m \ell \alpha (1-s)}} {B^{m\ell s}}\cdot \frac{2}{e^{n(\alpha+\epsilon)}}.
\end{align*}
Though the setting is slightly different, one can follow the proof of \eqref{eq:ml>n-1} and show that whenever $n$ is large enough,
\begin{equation}\label{eq:n<ml}
 n(1-\epsilon)	 \le m\ell \le n
\end{equation}
since $n=k+m\ell+\ell_0$ and $\ell$ is fixed.
Hence, using $2e^{n(\alpha+\epsilon)}\le c \le 3e^{n(\alpha+\epsilon)}$ we get
\[\begin{split}
	\frac{e^{m \ell \alpha (1-s)}} {B^{m\ell s}}\cdot \frac{2}{e^{n(\alpha+\epsilon)}}&= \frac{e^{n\alpha(1-s)}} {B^{ns}}\cdot \frac{1}{e^{n\alpha}} \cdot e^{O(n\epsilon)}=\frac{1} {e^{n\alpha s}B^{ns}}\cdot e^{O(n\epsilon)}\\
	&=\frac{1} {c^sB^{ns}}\cdot e^{O(n\epsilon)}.
\end{split}\]
Again, by decreasing $\epsilon$ and increasing $s$ if necessary, we have
\[\frac{1} {c^sB^{ns}}\cdot e^{O(n\epsilon)}\le \frac{1} {e^{2nt\epsilon}c^tB^{nt}}.\]
By \eqref{eq:qqq<-2} and \eqref{lowerdist-2} it follows that
\[\begin{split}
	\mu_2(\mathcal I_2(\tilde{\bm u},\bm a_1,\dots,\bm a_m,c))&\le\frac{q_{k+\ell_0+m \ell}(\tilde{\bm u},\bm a_1,\dots,\bm a_m)^{-2t}}{|I_{k+\ell_0}(\tilde {\bm u})|^{t}}\cdot \frac{1} {e^{2nt\epsilon}c^tB^{nt}}\\
	&\ll \frac{|\mathcal I_2(\tilde{\bm u},\bm a_1,\dots,\bm a_m,c)|^t}{|I_{k+\ell_0}(\tilde {\bm u})|^t }.\qedhere
\end{split}\]

\end{proof}
\begin{rem}\label{r:ineq}
	We present an inequality that will be also used in later discussion:
	\[\begin{split}
		&e^{n(\alpha+\epsilon)(1+t)}\cdot q_{n}(\tilde{\bm u},\bm a_1,\dots,\bm a_m)^{2t}\cdot \bigg(\prod_{i=1}^m\frac{e^{ \ell \alpha (1-s)}}{q_\ell(\bm a_i)^{2s}B^{\ell s}}\bigg)\cdot \frac{2}{e^{n(\alpha+\epsilon)}}\\
		\ll & e^{n(\alpha+\epsilon)(1+t)} \cdot\frac{1}{|I_{k+\ell_0}(\tilde {\bm u})|^t}\cdot \frac{e^{m\ell\alpha(1-s)}}{B^{m\ell s}}\cdot \frac{2}{e^{n(\alpha+\epsilon)}}\hspace{10em}\text{by \eqref{eq:qqq<-2}}\\
		=&\frac{1}{|I_{k+\ell_0}(\tilde {\bm u})|^t}\cdot\frac{e^{n\alpha(1+t-s)}} {B^{ns}}\cdot e^{O(n\epsilon)}\hspace{15em}\text{by \eqref{eq:n<ml}}\\
		\le & \frac{1}{|I_{k+\ell_0}(\tilde {\bm u})|^t}\cdot\frac{B^{ns(1+t-s)}} {B^{ns}}\cdot e^{O(n\epsilon)}\hspace{15em}\text{by \eqref{eq:ealpha<}}\\
		\le & \frac{1}{|I_{k+\ell_0}(\tilde {\bm u})|^t},
	\end{split}\]
	where the last inequality follows from the fact that the error term $e^{O(n\epsilon)}$ can be made smaller than $B^{ns(s-t)}$.
\end{rem}

Next, the following lemma gives the estimation of  the $\mu_2$-measure of any ball $B(x, r)$ with $x\in[0,1] $ and $r>0$.

\begin{lem}\label{l:holder2}
	For any $ r>0 $ and $ x\in[0,1] $, we have
	\[\mu_2(B(x,r))\ll (M+2)^4(M+1)^{2\ell}\cdot\frac{r^t}{|I_{k+\ell_0}(\tilde {\bm u})|^t},\]
	where the unspecified constant is absolute.
\end{lem}

\begin{proof}[Sketch proof]
	It suffices to focus on the case $ r<|I_{k+\ell_0}(\tilde{\bm u})| $. We need to consider three cases according to Lemma \ref{l:pro En-2}.

	\noindent \textbf{Case 1}: $r\le|I_{n+1}(\tilde{\bm u},\bm a_1,\dots,\bm a_m,c)|/18$. In view of Lemma \ref{l:pro En-2} (2), $B(x,r)$ only intersects the fundamental interval $\mathcal I_2(\tilde{\bm u},\bm a_1,\dots,\bm a_m,c)$ to which $x$ belongs. By the same reason as Case 1 in Lemma \ref{l:holder1}, and using Lemma \ref{l:meamu2} (2) instead of Lemma \ref{l:meamu} (2), we deduce that
	\[\begin{split}
		\mu_2(B(x,r))\ll\frac{r^t}{|I_{k+\ell_0}(\tilde {\bm u})|^t}.
	\end{split}\]
%	It follows that
%	\[\begin{split}
%		\lm^{(2)}_{\bm a_1,\dots,\bm a_m,c}(B(x,r))&\le\min\biggl\{1,\frac{r}{|\mathcal I_2(\tilde{\bm u},\bm a_1,\dots,\bm a_m,c)|}\biggr\}\\
%		&\le \frac{r^t}{|\mathcal I_2(\tilde{\bm u},\bm a_1,\dots,\bm a_m,c)|^t}.
%	\end{split}\]
%	This together with Lemma \ref{l:meamu2} (2) gives
%	\[\begin{split}
%		\mu_2(B(x,r))\le&\bigg(\prod_{i=1}^m\frac{e^{ \ell \alpha (1-s)}}{q_\ell(\bm a_i)^{2s}B^{\ell s}}\bigg)\cdot \frac{2}{e^{n(\alpha+\epsilon)}}\cdot \lm_{\bm a_1,\dots,\bm a_m,b}^{(2)}(B(x,r))\\
%		=&\mu_2({\mathcal I_2}(\tilde{\bm u},\bm a_1,\dots,\bm a_m,b))\cdot \lm_{\bm a_1,\dots,\bm a_m,c}^{(2)}(B(x,r))\\
%		\ll  & \frac{|\mathcal I_2(\tilde{\bm u},\bm a_1,\dots,\bm a_m,c)|^t}{|I_{k+\ell_0}(\tilde {\bm u})|^t }\cdot \frac{r^t}{|\mathcal I_2(\tilde{\bm u},\bm a_1,\dots,\bm a_m,c)|^t}\\
%		= & \frac{r^t}{|I_{k+\ell_0}(\tilde {\bm u})|^t}.
%	\end{split}\]

	\noindent \textbf{Case 2}: $|I_{n+1}(\tilde{\bm u},\bm a_1,\dots,\bm a_m,c)|/18< r\le|I_{n}(\tilde{\bm u},\bm a_1,\dots,\bm a_m)|/(2(M+2)^4)$. In this case, the ball $B(x,r)$ intersects exactly the $n$th level basic cylinder  $I_{n}(\tilde{\bm u},\bm a_1,\dots,\bm a_m)$, but may intersect multiple $(n+1)$th level basic cylinders inside this cylinder. Let
	\[\Delta_2(x)=\{2e^{n(\alpha+\epsilon)}\le c\le 3e^{n(\alpha+\epsilon)}: \text{$c$ is even and } I_{n+1}(\tilde{\bm u},\bm a_1,\dots,\bm a_m,c)\cap B(x,r)\ne\emptyset\}.\]
	In comparison to Case 2 in Lemma \ref{l:holder1}, here the choices for $c$ is at most $e^{n(\alpha+\epsilon)}/2$ and each $(n+1)$th level basic cylinder $I_{n+1}(\tilde{\bm u},\bm a_1,\dots,\bm a_m,c)$ is of length at least
	\[\ge 2^{-1}q_{n+1}(\tilde{\bm u},\bm a_1,\dots,\bm a_m,c)^{-2}\ge 2^ {-1} 2^{-2} 3^{-2}q_{n}(\tilde{\bm u},\bm a_1,\dots,\bm a_m)^{-2}e^{-2n(\alpha+\epsilon)}.\]
	This gives
	\[\begin{split}
		\# \Delta_2(x)&\ll\min\{e^{n(\alpha+\epsilon)}/2, r\cdot q_{n}(\tilde{\bm u},\bm a_1,\dots,\bm a_m)^{2}e^{2n(\alpha+\epsilon)}\}\\
		&\le e^{n(\alpha+\epsilon)(1-t)}\cdot(r\cdot q_{n}(\tilde{\bm u},\bm a_1,\dots,\bm a_m)^{2}e^{2n(\alpha+\epsilon)})^t\\
		&=e^{n(\alpha+\epsilon)(1+t)}\cdot r^t\cdot q_{n}(\tilde{\bm u},\bm a_1,\dots,\bm a_m)^{2t}.
	\end{split}\]
	By the definition of $\mu_2$ and inequality presented in Remark \ref{r:ineq},
	\[\begin{split}
		\mu_2(B(x,r))\le& \#\Delta_2(x)\cdot\max_{2e^{n(\alpha+\epsilon)}\le c\le 3e^{n(\alpha+\epsilon)}\atop c\text{ is even}} \mu_2(\mathcal I_2(\tilde{\bm u},\bm a_1,\dots,\bm a_m,c))\\
		\ll  &e^{n(\alpha+\epsilon)(1+t)}\cdot r^t\cdot q_{n}(\tilde{\bm u},\bm a_1,\dots,\bm a_m)^{2t}\cdot \bigg(\prod_{i=1}^m\frac{e^{ \ell \alpha (1-s)}}{q_\ell(\bm a_i)^{2s}B^{\ell s}}\bigg)\cdot \frac{2}{e^{n(\alpha+\epsilon)}}\\
		\le & \frac{r^t}{|I_{k+\ell_0}(\tilde {\bm u})|^t}.
	\end{split}\]

	\noindent \textbf{Case 3}: $|I_{k+\ell_0+(p+1) \ell}(\tilde{\bm u},\bm a_1,\dots,\bm a_{p+1})|/(2(M+2)^4) \le r< | I_{k+\ell_0+p\ell}(\tilde{\bm u},\bm a_1,\dots,\bm a_{p})|/(2($ $M+2)^4) $ for some $1\le p\le m-1$.  In this case, the ball $B(x,r)$ only intersects one $(k+\ell_0+p \ell)$th level basic cylinder, i.e. $I_{ k+\ell_0+p \ell}(\tilde{\bm u},\bm a_1,\dots,\bm a_p)$. Hence, following the same line as Case 3 in Lemma \ref{l:holder1}, we have
%	\[\begin{split}
%		\mu_2(B(x,r))&\le\mu(I_{k+\ell_0+p\ell}(\tilde{ \bm u},\bm a_1,\dots,\bm a_p))  \\
%		& \le \frac{q_{k+\ell_0+p \ell}(\tilde{\bm u},\bm a_1,\dots,\bm a_{p})^{-2t}}{| I_{k+\ell_0}(\tilde{\bm u})|^t} \\
%		& = \frac{q_{k+\ell_0+p \ell}(\tilde{\bm u},\bm a_1,\dots,\bm a_{p})^{-2t} }{| I_{k+\ell_0}(\tilde{\bm u})|^t}\times\frac{q_{\ell} (\bm a_{p+1})^{-2t}}{q_{\ell} (\bm a_{p+1})^{-2t}}\\
%		& \ll \frac{q_{k+\ell_0+(p+1) \ell}(\tilde{\bm u},\bm a_1,\dots,\bm a_{p+1})^{-2t}}{| I_{k+\ell_0}(\tilde{\bm u})|^t }\cdot(M+1)^{2\ell}\\
%		& \ll (M+1)^{2\ell}\cdot\frac{r^t}{| I_{k+\ell_0}(\tilde{\bm u})|^t}.
%	\end{split}\]
	\[\begin{split}
		\mu_2(B(x,r))&\ll(M+2)^4(M+1)^{2\ell}\cdot\frac{r^t}{| I_{k+\ell_0}(\tilde{\bm u})|^t}.
	\end{split}\]

	Combining the estimation given in Cases 1--3, we arrive at the conclusion.
\end{proof}
\begin{proof}[Completing the proof of Theorem \ref{t:main}] The proof is the same as that in the end of Section \ref{s:1}, we leave out the details.
\end{proof}

\subsection{Case III: $s^*=\limsup\limits_{n \to \infty} s_n$ is obtained along a subsequence of $\{s_{n,3}\}_{n\ge 1}$}\label{s:3}
There exist infinitely many $n$ such that
\begin{equation}\label{eq:sn-3}
	s_n=s_{n,3}>t.
\end{equation}
Fix an $\epsilon<s^*-t$. By the definition of $s_n$,
\[s_{n,3}> s_{n,2}.\]
With Lemma \ref{l:a1(zn)>} (3), it follows that
\[a_1(z_n)\le B^{ns_{n,3}/2}<B^{n/2}.\]
By taking a subsequence, assume that integers satisfying \eqref{eq:sn-3} will ensure the existence of the following limit
\[\lim_{n\text{ satisfies \eqref{eq:sn-3}}\atop n\to\infty}\frac{\log a_1(z_n)}{n}=: \beta \le \frac{\log B}{2}.\]
Recall the definition of $s_{n,3}$. By the continuity of the pressure functions, we can infer from the above discussion that
\begin{equation}\label{pressure3}
s^*=\inf\{s\in [0,1]:P(T, -s\log|T'|-(s/2)\log B-s\beta)\le 0\}.
\end{equation}

By Proposition \ref{p:property of s}, there exist integers $ \ell \ge 2t/\epsilon +1$ and $ M $ such that the unique positive number $s=s(\ell,M)$ satisfying the equation
\begin{equation}\label{eq:defs-3}
	\sum_{\bm a\in\{1,\dots,M \}^\ell}\frac{B^{-\ell s/2}}{q_\ell(\bm a)^{2s}e^{\beta\ell s}}=1,
\end{equation}
is greater than $t+\epsilon$.

Choose a large integer $n\in \N$ so that \eqref{eq:sn-3} is satisfied and
\begin{equation}\label{eq:condn3}
	n-k\ge \ell,\quad e^{n(\beta-\epsilon)}\le a_1(z_n)\le e^{n(\beta+\epsilon)}.
\end{equation}
Write $n-k=m\ell+\ell_0$, where $0\le \ell_0<\ell$.

Let $\tilde{\bm u}=(\bm u,1^{\ell_0})\in\N^{k+\ell_0}$ and consider the following set defined by $(n+1)$th level cylinders:
\begin{equation}\label{eq:basiccyl-3}
	\{I_{n+1}(\tilde {\bm u},\bm a_1,\dots,\bm a_m,a_1(z_n)): ~\bm a_i\in\{1,\dots,M\}^\ell, 1\le i\le m\}.
\end{equation}
We stress that for the $(n+1)$th position, there is only one choice, that is $a_1(z_n)$. This makes the discussion much easier than previous two.
By \eqref{eq:Jn3>3}, there is an interval, denoted by $\mathcal I_3(\tilde{\bm u},\bm a_1,\dots,\bm a_m,a_1(z_n))$ such that
\begin{equation}\label{eq:interval-3}
	\mathcal I_3(\tilde{\bm u},\bm a_1,\dots,\bm a_m,a_1(z_n))\subset F_n\cap I_{n+1}(\tilde{\bm u},\bm a_1,\dots,\bm a_m,a_1(z_n)).
\end{equation}
Additionally, since $a_1(z_n)\le B^{n/2}$, by \eqref{eq:Jn3>3}
\[|\mathcal I_3(\tilde{\bm u},\bm a_1,\dots,\bm a_m,a_1(z_n))|\ge\frac{B^{-n/2}}{4q_n(\tilde{\bm u},\bm a_1,\dots,\bm a_m)^2a_1(z_n)}.\]
%Following the same terminology in the last section, we still call $\mathcal I_2(\tilde{\bm u},\bm a_1,\dots,\bm a_m,c)$ in $F_n\cap I_{k+\ell_0}(\tilde{\bm u})$ as {\it fundamental interval}, and the cylinders $I_{n+1}(\tilde {\bm u},\bm a_1,\dots,\bm a_m,c)$ and $I_{k+\ell_0+p\ell}(\tilde{\bm u},\bm a_1,\dots,\bm a_p)$ as {\it basic cylinders}.

Define a probability measure $\mu_3$ supported on $F_n\cap I_{k+\ell_0}(\tilde{\bm u})$ as follows:
\[\mu_3=\sum_{\bm a_1\in\{1,\dots,M \}^\ell}\cdots\sum_{\bm a_m\in\{1,\dots,M \}^\ell}\bigg(\prod_{i=1}^m\frac{1}{q_\ell(\bm a_i)^{2s}e^{\beta\ell s} B^{\ell s/2}}\bigg)\cdot \lm^{(3)}_{\bm a_1,\dots,\bm a_m,a_1(z_n)},\]
where
\[\lm^{(3)}_{\bm a_1,\dots,\bm a_m,a_1(z_n)}:=\frac{\lm|_{\mathcal I_3(\tilde{\bm u},\bm a_1,\dots,\bm a_m,a_1(z_n))}}{\lm(\mathcal I_3(\tilde{\bm u},\bm a_1,\dots,\bm a_m,a_1(z_n)))}.\]

The next two lemmas describe the gap between fundamental intervals defined in \eqref{eq:interval-2} and the $\mu_3$-measures of these fundamental intervals, respectively. The first one follows the same lines as the proof of Lemma \ref{l:pro En-1}.
\begin{lem}\label{l:pro En-3}
	Let $ \mathcal I_3=\mathcal I_3(\tilde{\bm u},\bm a_1,\dots,\bm a_m,a_1(z_n)) $ and $ \mathcal I_3'=\mathcal I_3(\tilde{\bm u},\bm a_1',\dots,\bm a_m',a_1(z_n))$ be two intervals defined in \eqref{eq:interval-2}. If $ (\bm a_1,\dots,\bm a_{p-1})= (\bm a_1',\dots,\bm a_{p-1}') $ but $ \bm a_p \ne \bm a_p' $ for some $ 1\le p\le m $, then
	\[
	\dist(\mathcal I_3,\mathcal I_3')\ge\frac{|I_{k+\ell_0+p \ell}(\tilde{\bm u},\bm a_1,\dots,\bm a_p)|}{2(M+2)^4}.
	\]
\end{lem}

Instead of giving complete proof of the following lemma, we merely point out which changes have to be made.
\begin{lem}\label{l:meamu3}
	The following statements hold:
	\begin{enumerate}[\upshape(1)]
		\item For any $ (\tilde{\bm u},\bm a_1,\dots,\bm a_p) $ with $ 0\le p\le m $,
		\[\mu_3 (I_{k+\ell_0+p\ell}(\tilde{\bm u},\bm a_1,\dots,\bm a_p))\ll\frac{q_{k+\ell_0+p \ell}(\tilde{\bm u},\bm a_1,\dots,\bm a_p)^{-2t}}{|I_{k+\ell_0}(\tilde {\bm u})|^{t}}.\]
		\item For any $ (\tilde{\bm u},\bm a_1,\dots,\bm a_m,a_1(z_n))$,
		\[\mu_3(\mathcal I_3(\tilde{\bm u},\bm a_1,\dots,\bm a_m,a_1(z_n)))\ll \frac{|\mathcal I_3(\tilde{\bm u},\bm a_1,\dots,\bm a_m,a_1(z_n))|^t}{|I_{k+\ell_0}(\tilde {\bm u})|^t }.\]
	\end{enumerate}
\end{lem}

\begin{proof}[Sketch proof]
	(1) Following the same lines as in Lemma \ref{l:meamu} (1), one has
	\[\begin{split}
		\mu_3(I_{k+\ell_0+p\ell}(\tilde{\bm u},\bm a_1,\dots,\bm a_p))&=\prod_{i=1}^p\frac{1}{q_\ell(\bm a_i)^{2s}e^{\beta\ell s} B^{\ell s/2}}\le \prod_{i=1}^p\frac{1}{q_\ell(\bm a_i)^{2s}}\\
		&\ll \frac{q_{k+\ell_0+p \ell}(\tilde{\bm u},\bm a_1,\dots,\bm a_p)^{-2t}}{|I_{k+\ell_0}(\tilde {\bm u})|^{t}}.
	\end{split}\]

	(2) By the definition of $\mu_3$,
	\begin{align*}
		\mu_3(\mathcal I_3(\tilde{\bm u},\bm a_1,\dots,\bm a_m,a_1(z_n)))
		=&\prod_{i=1}^m\frac{1}{q_\ell(\bm a_i)^{2s}e^{\beta\ell s} B^{\ell s/2}}.
	\end{align*}
	In view of item (2), we only need to estimate $\prod_{i=1}^m(e^{\beta\ell s} B^{\ell s/2})^{-1}=e^{-m\ell \beta s} B^{-m\ell s/2}$. Though the setting is slightly different, one can follow the proof of \eqref{eq:ml>n-1} and show that whenever $n$ is large enough,
	\[  n(1-\epsilon) \le m\ell\le n,\]
	since $n=k+m\ell+\ell_0$ and $\ell$ is fixed. Hence, by \eqref{eq:condn3},
	\[\begin{split}
	e^{-m\ell\beta s}B^{-m\ell s/2}=e^{-n\beta s}B^{-ns/2}\cdot e^{O(n\epsilon)}=a_1(z_n)^{-s}B^{-ns/2}\cdot e^{O(\epsilon)}.
	\end{split}\]
	Since  $a_1(z_n)=e^{n(\beta+O(\epsilon))}$, by decreasing $\epsilon$ if necessary, we have
	\[a_1(z_n)^{-s}B^{-ns/2}\cdot e^{O(\epsilon)}\le a_1(z_n)^{-t}B^{-nt/2}.\]
	This together with item (1) yields the conclusion.
\end{proof}

Next, the following lemma gives the estimation of  the $\mu_3$-measure of arbitrary ball $B(x, r)$ with $x \in[0,1]$ and $r>0$.

\begin{lem}\label{l:holder3}
	For any $ r>0 $ and $ x\in[0,1] $, we have
	\[\mu_2(B(x,r))\ll (M+2)^4(M+1)^{2\ell}\cdot \frac{r^t}{|I_{k+\ell_0}(\tilde {\bm u})|^t}.\]
\end{lem}

\begin{proof}
	Assume that $ x\in \mathcal I_3(\tilde{\bm u},\bm a_1,\dots,\bm a_m,a_1(z_n)) $ and $r>0$. To estimate $\mu_3(B(x,r))$, compared to Lemmas \ref{l:holder1} and \ref{l:holder2}, only two cases need to be considered instead of three, because there is only one choice for the $(n+1)$th position of the basic cylinder (see \eqref{eq:basiccyl-3}), namely $a_1(z_n)$. According to Lemma \ref{l:pro En-3}, the proof is split into two cases:

	\noindent \textbf{Case 1}: $r\le\dfrac{|I_{k+\ell_0+m \ell}(\tilde{\bm u},\bm a_1,\dots,\bm a_{p+1})|}{2(M+2)^4}$.

	\noindent \textbf{Case 2}: $\dfrac{|I_{k+\ell_0+(p+1) \ell}(\tilde{\bm u},\bm a_1,\dots,\bm a_{p+1})|}{2(M+2)^4}\le r< \dfrac{|I_{k+\ell_0+p \ell}(\tilde{\bm u},\bm a_1,\dots,\bm a_{p})|}{2(M+2)^4}$ for some $1\le p\le m-1$.

	The argument is similar to Cases 1 and 3 in Lemma \ref{l:holder1} (or Cases 1 and 3 in Lemma \ref{l:holder2}), respectively. We omit the details.
\end{proof}
\begin{proof}[Completing the proof of Theorem \ref{t:main}] The proof is the same as that in the end of Section \ref{s:1}, we leave out the details.
\end{proof}

 {\bf Acknowledgement.}
This work is supported by
the {\color{red}Fundamental} Research {\color{red}Funds} for the Central Universities (No. SWU-KQ24025).

\bibliographystyle{amsplain}

\end{document}